\documentclass[a4paper,12pt]{amsart}
\usepackage{amssymb}
\usepackage[alwaysadjust]{paralist}
\usepackage[capitalize]{cleveref}

\usepackage[pdftex]{graphicx}
\numberwithin{equation}{section}

\newtheorem{cor}[equation]{Corollary}
\newtheorem{lem}[equation]{Lemma}
\newtheorem{prop}[equation]{Proposition}
\newtheorem{thm}[equation]{Theorem}
\theoremstyle{definition}

\newtheorem{exa}[equation]{Example}

\newtheorem{rem}[equation]{Remark}
\newtheorem{rems}[equation]{Remarks}

\def\R{\mathbb R}

\def\Z{\mathbb Z}

\def\ve{\varepsilon}
\def\vf{\varphi}
\def\la{\langle}
\def\ra{\rangle}
\newcommand{\cs}{\operatorname{cs}}
\newcommand{\ct}{\operatorname{ct}}
\newcommand{\diam}{\operatorname{diam}}

\newcommand{\ess}{\operatorname{ess}}

\newcommand{\Or}{\operatorname{Or}}

\newcommand{\sn}{\operatorname{sn}}

\newcommand{\supp}{\operatorname{supp}}
\newcommand{\sys}{\operatorname{sys}}
\newcommand{\tn}{\operatorname{tn}}

\newcommand{\vol}{\operatorname{vol}}
\DeclareMathOperator{\arsinh}{arsinh}

\begin{document}


\title[On the analytic systole of Riemannian surfaces]
{On the analytic systole of Riemannian surfaces of finite type}

\author{Werner Ballmann}
\address
{WB: Max Planck Institute for Mathematics,
Vivatsgasse 7, 53111 Bonn
and Hausdorff Center for Mathematics,
Endenicher Allee 60, 53115 Bonn.
}
\email{hwbllmnn\@@mpim-bonn.mpg.de}
\author{Henrik Matthiesen}
\address
{HM: Max Planck Institute for Mathematics,
Vivatsgasse 7, 53111 Bonn.}
\email{hematt\@@mpim-bonn.mpg.de}
\author{Sugata Mondal}
\address
{SM: Indiana University,
Rawles Hall, 831 E 3rd Street,
Bloomington, Indiana
and Max Planck Institute for Mathematics,
Vivatsgasse 7, 53111 Bonn}
\email{sumondal\@@iu.edu}

\thanks{\emph{Acknowledgments.}
We would like to thank the Max Planck Institute for Mathematics in Bonn,
the Hausdorff Center for Mathematics in Bonn,
and Indiana University at Bloomington
for their support and hospitality.
We would also like to thank the referees whose criticism led to a better presentation in the paper.}

\date{\today}


\begin{abstract}
In \cite{BMM} we introduced, for a Riemannian surface $S$,
the quantity $ \Lambda(S):=\inf_F\lambda_0(F)$,
where $\lambda_0(F)$ denotes the first Dirichlet eigenvalue
of $F$ and the infimum is taken over all compact subsurfaces $F$
of $S$ with smooth boundary and abelian fundamental group.
A result of Brooks \cite{Br2} implies $\Lambda(S)\ge\lambda_0(\tilde{S})$,
the bottom of the spectrum of the universal cover $\tilde{S}$.
In this paper, we discuss the strictness of the inequality.
Moreover, in the case of curvature bounds, we relate $\Lambda(S)$ with the systole,
improving the main result of \cite{Mo1}.
\end{abstract}

\maketitle


\section{Introduction}

Small eigenvalues of Riemannian surfaces, in particular of hyperbolic surfaces,
have been of interest in different mathematical fields for a long time.
Buser and Schmutz conjectured that a hyperbolic metric on the closed surface $S=S_g$
of genus $g\ge2$ has at most $2g-2$ eigenvalues below $1/4$ \cite{Bu1,S2}.
In \cite{OR},
Otal and Rosas proved a generalized version of this conjecture.
They showed that a real analytic Riemannian metric on $S_g$ with negative curvature
has at most $2g-2$ eigenvalues $\le\lambda_0(\tilde S)$,
where $\tilde S$ denotes the universal covering surface of $S$,
endowed with the lifted Riemannian metric,
and where $\lambda_0(\tilde S)$ denotes the bottom of the spectrum of $\tilde S$.
Recall here that,
for a Riemannian surface $F$ (possibly not complete)
with piecewise smooth boundary $\partial F$ (possibly empty),
the \emph{bottom of the spectrum of $F$} is defined to be
\begin{equation}
  \lambda_0 = \lambda_0(F) = \inf R(\varphi),
\end{equation}
where $\varphi$ runs over all non-vanishing smooth functions on $F$
with compact support in the interior $\mathring F=F\setminus\partial F$ of $F$
and where $R(\varphi)$ denotes the Rayleigh quotient of $\varphi$.
It is well known that the bottom of the spectrum
of the Euclidean and hyperbolic plane is $0$ and $1/4$, respectively.
When $F$ is closed, $\lambda_0(F)=0$,
when $F$ is compact and connected with non-empty boundary,
$\lambda_0(F)$ is the first Dirichlet eigenvalue of $F$.
In the latter case, $\lambda_0(F)$-eigenfunctions of $F$ do not have zeros in $\mathring F$
and, therefore, the multiplicity of $\lambda_0(F)$ as an eigenvalue of $F$ is one.
We then call the corresponding positive eigenfunction of $F$
with $L^2$-norm one the \emph{ground state} of $F$.

For a Riemannian surface $S$, with or without boundary,
we define the \emph{analytic systole} to be the quantity
\begin{equation}\label{anasy}
 \Lambda(S) = \inf_F\lambda_0(F),
\end{equation}
where the infimum is taken over all subsurfaces $F$ in $\mathring S$ with smooth boundary
which are diffeomorphic to a closed disc, annulus, or cross cap.
(A cross cap is frequently also called a M\"obius strip.)
Note that the fundamental groups of disc, annulus, and cross cap are cyclic,
hence amenable.
By the work of Brooks, we therefore have
\begin{equation}\label{brooks}
  \Lambda(S)\ge\lambda_0(\tilde{S})
\end{equation}
for all complete and connected Riemannian surfaces $S$,
see \cite[Theorem 1]{Br2} and also \cref{brooks2} below. 
The strictness of this inequalilty and other estimates of $\Lambda(S)$ are the topics of this article.

To clarify our terminology, a surface is a smooth manifold of dimension two.
A Riemann surface is a surface together with a conformal structure.
They are not the topic of this article.
We study Riemannian surfaces, that is, surfaces together with a Riemannian metric.

We say that a surface $S$ is of \emph{finite type} if its Euler characteristic $\chi(S)$
is finite and its boundary is compact (possibly empty).
It is well known that a connected surface $S$ is of finite type if and only if $S$
can be obtained from a closed surface by deleting a finite number of pairwise disjoint
points and open discs.

After first extensions of the results of Otal and Rosas in \cite{Ma} and \cite{Mo1},
we showed in \cite{BMM} and \cite{BMM1} that any complete Riemannian metric
on a connected surface $S$ of finite type with $\chi(S)<0$
has at most $-\chi(S)$ eigenvalues $\le\Lambda(S)$,
where the eigenvalues are understood to be Dirichlet eigenvalues if $\partial S\ne\emptyset$.
This result explains the significance of the analytic systole
and the interest in establishing strictness in \eqref{brooks}.

\subsection{Statement of main results}
In our first three results, we discuss the strictness of \eqref{brooks}.

\begin{thm}\label{main1}
If $S$ is a compact and connected Riemannian surface
whose fundamental group is not cyclic,
then $\Lambda(S)>\lambda_0(\tilde S)$.
\end{thm}

Note that the compact and connected surfaces with cyclic fundamental group
are precisely sphere, projective plane, closed disc, closed annulus, and closed cross cap.
For these, we always have equality in \eqref{brooks} as we will see in \cref{proqua}.

Recall that the spectrum of $S$ is discrete if $S$ is compact.
In general,
the spectrum of $S$ is the disjoint union of its discrete and essential parts,
where $\lambda\in\R$ belongs to the essential spectrum of $S$
if $\Delta-\lambda$ is not a Fredholm operator.
The \emph{bottom of the essential spectrum} is given by
\begin{equation}\label{boess}
  \lambda_{\ess}(S) = \lim_K\lambda_0(S\setminus K),
\end{equation}
where $K$ runs over the compact subsets of $S$, ordered by inclusion, see \cref{esp}.
By domain monotonicity,
the limit is monotone with respect to the ordering of the compact subsets of $S$.
If $S$ is compact, then $\lambda_{\ess}(S)=\infty$.
If $S$ is a complete and connected Riemannian surface of finite type, then
\begin{equation}\label{boess2}
  \lambda_{\ess}(S) \ge \Lambda(S)
\end{equation}
since the the ends of $S$ admit neighborhoods in $S$
whose connected components are diffeomorphic to open annuli.
Such neighborhoods of the ends of $S$ will be called \emph{cylindrical}.

\cref{main1} extends in the following way to surfaces of finite type,
be they compact or non-compact.

\begin{thm} \label{main3}
If $S$ is a complete and connected Riemannian surface of finite type
whose fundamental group is not cyclic,
then $\Lambda(S)>\lambda_0(\tilde S)$
if and only if $\lambda_{\ess}(S)>\lambda_0(\tilde S)$.
\end{thm}

For complete and connected Riemannian surfaces of finite type,
$\Lambda(S)$ is always between $\lambda_0(\tilde S)$ and $\lambda_{\ess}(S)$,
by \eqref{brooks} and \eqref{boess2}.
Hence the condition $\lambda_{\ess}(S)>\lambda_0(\tilde S)$ in \cref{main3}
is obviously necessary to have the strict inequality $\Lambda(S)>\lambda_0(\tilde S)$.
The hard part of the proof of \cref{main3} is to show that the condition is also sufficient.

\begin{rem}
Another way of stating the condition $\lambda_{\ess}(S)>\lambda_0(\tilde S)$ in \cref{main3}
is to require that there is a compact subset $K$ of $S$
such that $\lambda_0(S\setminus K)>\lambda_0(\tilde{S})$.
By domain monotonicity, this condition is then also satisfied for any compact subset $K'$
of $S$ containing $K$.
\end{rem}

\begin{exa}\label{exahype}
Let $S$ be a non-compact connected surface of finite type
whose fundamental group is not cyclic.
Using a decomposition of $S$ into pairs of pants,
it becomes obvious that $S$ carries complete hyperbolic metrics
with (possibly empty) geodesic boundary.
For any such metric, we have
\begin{equation*}
  \lambda_{\ess}(S) = \Lambda(S) = \lambda_0(\tilde S) = 1/4
\end{equation*}
since the ends of $S$ are then hyperbolic cusps or funnels. 

Tempted by this equality we investigate how {\it generic} this equality is among all smooth complete metrics on a non-compact surface of finite type. 
Our next result is that it is in fact rare.
Some form of rigidity in the case of equality would of course be very interesting; compare with \cref{suseq}.
\end{exa}

\begin{prop}\label{proqua}
Let $S$ be a connected surface of finite type. \\
\begin{inparaenum}[1)]
\item\label{proqua2}
If the fundamental group of $S$ is cyclic, then $\Lambda(S)=\lambda_0(\tilde S)$
for any complete Riemannian metric on $S$. \\
\item\label{proqua3}
If $S$ is non-compact and the fundamental group of $S$ is not cyclic,
then $S$ carries complete Riemannian metrics such that $\lambda_{\ess}(S)>\Lambda(S)$.
Moreover, if $\chi(S)<0$, then such metrics may be chosen to have curvature $K\le-1$ and finite or infinite area.\\
\item\label{proqua4}
If $\lambda_0(\tilde S, \tilde g) > 0$ for some Riemannian metric $g$ on $S$,
then a generic complete Riemannian metric $g'$ on $S$ in any neighborhood of $g$
in the uniform $C^\infty$ topology satisfies the strict inequality $\lambda_{\ess}(S,g')>\Lambda(S, g')$.
\end{inparaenum}
\end{prop}

By \cref{main3}, $\lambda_{\ess}(S)>\Lambda(S)$ implies $\Lambda(S)>\lambda_0(\tilde S)$.

In our next result,
we generalize the main result of the third author in \cite{Mo1},
which asserts that a hyperbolic metric on the closed surface $S_g$ of genus $g\ge2$
has at most $2g-2$ eigenvalues $\le1/4+\delta$, where
\begin{equation*}
  \delta = \min\{\pi/|S|,\sys(S)^2/|S|^2\}.
\end{equation*}
Here $|S|$ denotes the area of $S$ and $\sys(S)$, the \emph{systole} of $S$,
is defined to be the minimal possible length of an essential closed curve in $S$.

\begin{thm}\label{main2}
For a closed Riemannian surface $S$ with curvature $K\le\kappa\le0$, we have
\begin{equation*}
  \Lambda(S) \ge -\frac{\kappa}4 + \frac{\sys(S)^2}{|S|^2}.
\end{equation*}
\end{thm}

\begin{rems}
\begin{inparaenum}[1)]
\item
For closed Riemannian surfaces $S$ with curvature $K\le\kappa<0$,
we know in general only that $\lambda_0(\tilde{S})\ge-\kappa/4$.
Therefore \cref{main2} may not imply the strict inequality $\Lambda(S)>\lambda_0(\tilde{S})$
for such $S$.
In fact, the relation between $\lambda_0(\tilde{S})$ and the right hand side in \cref{main2}
is not clear.
Our method of proof, involving isoperimetric inequalities and Cheeger's inequality,
does not seem to be sophisticated enough to capture the difference between them.

\item
The proof of \cref{main2} also applies to non-compact surfaces of finite type.
In this case one needs to define the systole as the infimum over \emph{all} homotopically
non-trivial curves, not only the essential (not homotopic to a boundary component or a puncture) ones.
For this reason, the corresponding statement is not really interesting anymore.
If $|S|<\infty$, then $\sys(S)=0$ (by a refinement of the isosystolic inequality), 
and if $|S|=\infty$, then $\sys(S)/|S|=0$.

\item
The difference $\Lambda(S)-\lambda_0(\tilde S)$ can not be estimated
from below by a positive constant, which only depends on the topology
and the area of $S$.
In fact, given any $\ve>0$ and natural number $n$,
if the metric on $S$ is hyperbolic with sufficiently small systole,
then $\lambda_n(S)<1/4+\ve$, by \cite[Satz 2]{Bu1} or the proof of Theorem 8.1.2 in \cite{Bu2}.
\end{inparaenum}
\end{rems}

One may view \cref{main2} also as an upper bound on the systole
in terms of a curvature bound and $\Lambda(S).$
Together with our next result, this explains the name \emph{analytic systole}.

For a closed Riemannian surface $S$, we say that a closed geodesic $c$ of $S$
is a \emph{systolic geodesic} if it is essential with length $L(c)=\sys(S)$.
Clearly, systolic geodesics are simple.

\begin{thm}\label{esansy}
If $S$ is a closed Riemannian surface with $\chi(S)<0$ and curvature $K\ge-1$, then
\begin{equation*}
  \Lambda(S) \le \frac{1}{4} + \frac{4\pi^2}{w^2},
\end{equation*}
where
\begin{equation*}
 w = w(\sys(S)) =
 \begin{cases}
 \arsinh(1/\sinh(\sys(S)/2))) \\
 \arsinh(1/\sinh(\sys(S))) \\
 \end{cases}
\end{equation*}
if $S$ has a two-sided systolic geodesic
or if all systolic geodesics of $S$ are one-sided, respectively.
\end{thm}

Here we say that a simple closed curve in $S$ is two-sided or one-sided
if it has a tubular neigborhood which is diffeomorphic to an annulus or a cross cap, respectively.

Combining \cref{main2} and \cref{esansy},
we get that, for hyperbolic metrics,
$\Lambda(S)$ is squeezed between two functions of the systole.

\begin{cor}\label{syshype}
For closed hyperbolic surfaces, we have
\begin{equation*}
  \frac{1}4 + \frac{\sys(S)^2}{4\pi^2 \chi(S)^2}
  \le \Lambda(S)
  \le \frac{1}4 + \frac{4\pi^2}{w^2}
\end{equation*}
with $w=w(\sys(S))$ as in \cref{esansy}.
\end{cor}

Recall that $\arsinh x=\ln(x+\sqrt{x^2+1})$.
In particular, we have
\begin{equation*}
  w(\sys(S)) \sim -\ln(\sys(S)) \to \infty
  \quad\text{as $\sys(S)\to0$.}
\end{equation*}
We conclude that the analytic systole of hyperbolic metrics on closed surfaces tends to $1/4$
if and only if their systole tends to $0$.

\subsection{Main problems and arguments}
The only surfaces $S$ in \cref{main1} and \cref{main3} with Euler characteristic $\chi(S)\ge0$
are torus and Klein bottle.
For these, the proof of the inequality $\Lambda(S)>\lambda_0(\tilde S)$ is quite elementary.
The proof of the hard direction of \cref{main3},
namely establishing the strict inequality $\Lambda(S)>\lambda_0(\tilde S)$
under the condition $\lambda_{\ess}(S)>\lambda_0(\tilde S)$,
is rather involved in the case $\chi(S)<0$.

The domain monotonicity of the first Dirichlet eigenvalue implies that $\Lambda(S)$
can not be realized by any compact subsurface $F\subseteq S$ diffeomorphic to a disc,
an annulus or a cross cap. Keeping this in mind,
our general strategy for the proof of \cref{main3} is to show that the equality
$\Lambda(S)=\lambda_0(\tilde{S})$ would imply the existence of a non-trivial
$\lambda_0(\tilde{S})$-eigenfunction $\tilde\vf$ on $\tilde{S}$ or an appropriate cyclic
quotient $\hat{S}$ of $\tilde{S}$ that vanishes on an open set.

The condition $\lambda_{\ess}(S)>\lambda_0(\tilde S)$ forces a subsurface $F$ with $\lambda_0(F)$ close to 
$\lambda_0(\tilde S)$ to stay almost completely in a large compact set
in a weighted sense, the weight being the ground state. One then works essentially within a fixed
compact subsurface of $S.$ Two main problems that we still have to overcome in establishing
the existence of $\tilde\vf$ as above are
\begin{compactenum}[1)]
\item
a priori non-existence of a fixed quotient of $\tilde S$ along a sequence of 
subsurfaces approximating $\lambda_0(\tilde{S})$ and
\item
the absence of the compact Sobolev embedding $H^1 \hookrightarrow L^2$
on these covering spaces.
\end{compactenum}

As for the first problem,
the case of cross caps can be reduced to the case of annuli by considering the two-sheeted
orientation covering of the original surface.
The case of annuli is tackled by showing that only finitely many isotopy types of annuli
have the bottom of their spectrum close to $\lambda_0(\tilde{S})$.
A keystone of the argument is \cref{lambdad} which relates the bottom of the spectrum
of compact surfaces $F$ with the sum of the lengths of shortest curves in the free
homotopy classes of the boundary circles of $F$.

To tackle the second problem, we establish, in \cref{cheer},
an inradius estimate for superlevel sets of suitably truncated
ground states of a sequence of subsurfaces $F_n$
approximating $\lambda_0(\tilde{S})$.
The inradius estimate is proved by means of isoperimetric inequalities,
extending arguments from the proof of the Cheeger inequality.

\subsection{Structure of the article}\label{artstruc}
In \cref{secii},
we collect the relevant facts about isoperimetric inequalities on Riemannian surfaces.
In \cref{secci}, we extend Osserman's refined version of the Cheeger inequality
\cite[Lemma 1]{Os} for plane domains to compact Riemannian surfaces with boundary.
We also recall Osserman's elegant proof since we will need consequences and extensions
of his arguments.
The isoperimetric inequalities from \cref{c1} and the Cheeger inequality are then used
in \cref{secq} to obtain a generalized version of \cref{main2}.
The arguments here are very much in the spirit of Osserman \cite{Os} and Croke \cite{Cr}.
As an application of our discussion, we obtain \cref{main3} for the case
where $S$ is a torus or a Klein bottle.
This section closes with the proof of \cref{esansy},
which involves methods which are different from those of the rest of the article.
Sections \ref{secfe} and \ref{secfe2} are concerned with properties of the ground states
of compact Riemannian surfaces with boundary.
The main objectives are \cref{lambdad} on the relation of the bottom of the spectrum
to other geometric quantities and \cref{cheer} on the inradius of superlevel sets
of ground states.
In \cref{sece}, we complete the proof of \cref{main3}.
\cref{secr} contains the proof of \cref{proqua} and some remarks and questions.
In particular, we draw attention to problems in optimal design
which are related to optimal estimates of the analytic systol.
In \cref{appendix}, we discuss an extension of the result of Brooks
quoted in connection with \eqref{brooks}.

\section{Isoperimetric inequalities}\label{secii}

The content of the present section is related to and extends Lemma 1 of \cite{Os}
in the way we will need it.

Let $F$ be a compact and connected surface with piecewise smooth boundary $\partial F\ne\emptyset$
and interior $\mathring F=F\setminus\partial F$.
The components of $\partial F$ are piecewise smooth circles.
Denote by $\chi=\chi(F)$ the Euler characteristic of $F.$

Assume that $F$ is endowed with a Riemannian metric
and denote by $K$ the Gauss curvature of $F$.
Let $|F|$ and $|\partial F|$ be the area of $F$ and the length of $\partial F$, respectively,
and
\begin{equation}
  \rho=\rho_F=\max\{d(x,\partial F)\mid x\in F\}
\end{equation}
be the \emph{inradius} of $F$.

For a function $f \colon F \to \R$, write $f^+=\max(f,0)$ for its positive part.
We recall the following isoperimetric inequalities.

\begin{thm}\label{l1}
For any $F$ as above and $\kappa\in\R$, we have
\begin{align}
  |\partial F|^2 &\ge -\kappa |F|^2 + 2\left(2\pi\chi - \int_F(K-\kappa)^+dx \right) |F|.
  \label{l10}\tag{1} 
\intertext{If $\kappa\le0$, then}
  |\partial F| &\ge |F|\ct_{\kappa}\rho
  + \left(2\pi\chi - \int_F(K-\kappa)^+dx \right)\tn_{\kappa}\frac\rho2.
  \label{l1a}\tag{2}
\end{align}
If $F$ is not a disc and $\kappa<0$, then
\begin{equation}
  (|\partial F|^2-\ell^2)^{1/2} \ge \sqrt{-\kappa} |F|
  + \frac1{\sqrt{-\kappa}}\left(2\pi\chi - \int_F(K-\kappa)^+dx \right).
  \label{l1b}\tag{3}
\end{equation}
where $\ell$ denotes the sum of the lengths of the shortest loops
in the free homotopy classes (in $F$) of the boundary circles of $F$.
\end{thm}

In the second inequality, $\tn_{\kappa}=\sn_{\kappa}/\cs_{\kappa}$
and $\ct_{\kappa}=\cs_{\kappa}/\sn_{\kappa}$,
where $\sn_{\kappa}$ and $\cs_{\kappa}$ are the solutions
of the differential equation $\ddot u+\kappa u=0$ with respective initial conditions
\begin{equation*}
  \sn_{\kappa}(0)=0, \, \sn_{\kappa}'(0)=1
  \quad\text{and}\quad
  \cs_{\kappa}(0)=1, \, \cs_{\kappa}'(0)=0.
\end{equation*}
The first inequality of \cref{l1} corresponds to \cite[Theorem 2.2.1]{BZ},
the third to the (outer) inequality in (20) of \cite[p.\,15]{BZ}.
We added \lq\lq in $F$\rq\rq\ in parentheses in the statement since we will use \cref{l1}
in the case where $F$ is a domain in a surface $S$.
Then the length of a shortest loop in the free homotopy class in $S$
of a boundary circle $c$ of $F$ might be smaller than the length of a shortest loop
in the free homotopy class of $c$ in $F$.

\begin{proof}[Proof of \cref{l1}\,(\ref{l1a})]
We apply \cite[Theorem 2.4.2]{BZ} in the case $t=\rho$.
The function $f=f(t)$ of \cite{BZ} measures the area of the collar of width $t$ about $\partial F$
and, therefore, we have $f(\rho)=|F|$ by the definition of $\rho$.
The function $a=a(t)$ of \cite{BZ} satisfies
\begin{equation*}
  a(\rho) = \kappa|F| - 2\pi\chi + \int_F(K-\kappa)^+dx.
\end{equation*}
In our notation, the function $\psi$ of \cite{BZ} is given by
\begin{equation*}
  \psi(t) = a(t)\frac{1-\cs_\kappa t}{\kappa} + |\partial F| \sn_\kappa t,
\end{equation*}
where we set $(1-\cs_\kappa t)/\kappa=t^2/2$ for $\kappa=0$.
Now Theorem 2.4.2 of \cite{BZ}, in the case $\kappa\le0$ and $t=\rho$,
asserts that $ f(\rho) \le \psi(\rho)$, that is, that
\begin{equation*}
  |F| \le  |F|(1-\cs_\kappa\rho)
  - \left( 2\pi\chi - \int_F(K-\kappa)^+dx \right)\frac{1-\cs_\kappa\rho}{\kappa}
  + |\partial F|\sn_{\kappa}\rho.
\end{equation*}
Therefore we get
\begin{equation*}
  |\partial F| \sn_{\kappa}\rho
  \ge |F|\cs_{\kappa}\rho + \left(2\pi\chi
  - \int_F(K-\kappa)^+dx \right)\frac{1-\cs_{\kappa}\rho}{\kappa}.
\end{equation*}
This implies \eqref{l1a}
since $(1-\cs_\kappa t)/\kappa\sn_\kappa t = \tn_\kappa(t/2)$.
\end{proof}

\begin{cor}\label{c1}
If $K\le\kappa$, then we have:
\begin{compactenum}[1)]
\item\label{c10}
If $F$ is a disc, then $|\partial F|^2\ge-\kappa |F|^2+4\pi |F|$.
\item\label{c1a}
If $\chi\ge0$ and $\kappa\le0$,
then $|\partial F| \ge |F|\ct_{\kappa}\rho$.
\item\label{c1b}
If $\chi=0$ and $\kappa\le0$,
then $|\partial F|^2 \ge - \kappa|F|^2 +\ell^2$.
\end{compactenum}
\end{cor}

Note that we always have $|\partial F|^2 \ge\ell^2$,
by the definition of $\ell$.

\section{Cheeger inequality revisited}\label{secci}

In Lemma 2 of \cite{Os},
Osserman discusses a refinement of the Cheeger inequality for compact planar domains,
endowed with Riemannian metrics.
We will need an extension of Osserman's Lemma 2.

As above, we let $F$ be a compact and connected Riemannian surface $F$
with piecewise smooth boundary $\partial F\ne\emptyset$. 
The \emph{Cheeger constant} of $F$ is defined to be the number
\begin{equation*}
  h=h(F)=\inf  |\partial F'|/|F'|,
\end{equation*}
where the infimum is taken over all compact subsurfaces $F'$ of $\mathring F$
with smooth boundary.
Note that closed surfaces cannot occur as subsurfaces $F'$ of $F$
since $F$ is connected with non-empty boundary.

\begin{lem}\label{cheex}
The Cheeger constant is given by
\begin{equation*}
  h = \inf |\partial F'|/|F'|,
\end{equation*}
where the infimum is taken over all compact and connected subsurfaces $F'$ of $\mathring F$
with smooth boundary such that the boundary of each component of $F\setminus F'$
has at least one boundary circle in $\mathring F$
and contains at least one boundary circle of $F$.
\end{lem}

\begin{proof}
Let $F'$ be a compact subsurface of $\mathring F$ with smooth boundary
and denote by $F'_1,\dots,F'_k$ the components of $F'$.
Then
\begin{equation*}
  |F'|=|F'_1|+\dots+|F'_k|\quad\text{and}\quad |\partial F'|=|\partial F'_1|+\dots+|\partial F'_k|
\end{equation*}
and hence
\begin{equation*}
 \inf\frac{|\partial F'_i|}{|F'_i|} \le \frac{|\partial F'_1|+\dots+|\partial F'_k|}{|F'_1|+\dots+|F'_k|} = \frac{|\partial F'|}{|F'|}.
\end{equation*}
This shows that the infimum $h$ can be taken over compact and connected
subsurfaces $F'$ of $\mathring F$ with smooth boundary.

Let now $C$ be a component of $F\setminus F'$.
Suppose first that the boundary of $C$ does not contain a boundary circle of $F$.
Then $F''=F'\cup C$ is a compact and connected subsurface of $\mathring F$
with area $|F''|>|F'|$ and length of boundary $|\partial F''|<|\partial F'|$.
It follows that the infimum $h$ is attained by compact and connected subsurfaces
$F'$ of $\mathring F$ with smooth boundary such that the boundary
of each component of $F\setminus F'$ contains a boundary circle of $F$.

If the boundary of $C$ would not have a boundary circle in $\mathring F$,
then $C$ would have to coincide with $F$ since $F$ is connected.
But then $F'$ would be empty, a contradiction.
\end{proof}

By a slight variation of the standard  terminology,
we say that a subsurface $S$ of a surface $T$ is \emph{incompressible in $T$}
if the induced maps of fundamental groups are injective,
for all connected component $C$ of $S$.
In particular, embedded discs are always incompressible.

\begin{prop}\label{cheey}
The Cheeger constant is given by
\begin{equation*}
  h = \inf |\partial F'|/|F'|,
\end{equation*}
where the infimum is taken over all incompressible compact and connected
subsurfaces $F'$ of $\mathring F$ with smooth boundary
such that no component of $F\setminus F'$ is a disc or a cross cap.
Any such $F'$ satisfies $\chi(F')\ge\chi(F)$ with equality if and only
if $F\setminus F'$ is a collared neighborhood of $\partial F$,
consisting of annuli about the boundary circles of $F$.
\end{prop}

For example, if $F$ is an annulus,
then we only need to consider discs and incompressible annuli $F'$ in $F$;
if $F$ is a cross cap, then only discs and incompressible annuli and cross caps $F'$.

\begin{proof}[Proof of \cref{cheey}]
By \cref{cheex}, the Cheeger constant $h$ is realized by compact and connected
subsurfaces $F'$ of $\mathring F$ with smooth boundary such that each component
of $F\setminus F'$ has at least two boundary circles.
This excludes discs and cross caps as components of $F\setminus F'$.
We have
\begin{equation*}
  \chi(F) = \chi(F') + \chi(F\setminus\mathring F') \le \chi(F')
\end{equation*}
since the intersection of $F'$ with $F\setminus\mathring F'$ consists of circles
and since no component of $F\setminus\mathring F'$ is a disc.
Furthermore, equality can only occur if $\chi(F\setminus\mathring F')=0$.
By what we already know, this can only happen if the components
of $F\setminus\mathring F'$ are annuli.
By \cref{cheex} and since $F'\subseteq\mathring F$,
they constitute a collared neighborhood of $\partial F$.

It remains to show the incompressibility of $F'$.
If this would not hold,
$F'$ would contain a Jordan loop $c$ which is not contractible in $F'$,
but is contractible in $\mathring F$.
Then $c$ would be the boundary of an embedded disc $D$ in $\mathring F$
which is not contained in $F'$.
Since $\partial D\subseteq F'$,
$D\setminus F'$ would consist of components of $F\setminus F'$.
Their boundary would be in $D\subseteq\mathring F$ in contradiction to \cref{cheex}.
\end{proof}

Recall the classical Cheeger inequality.

\begin{thm}[Cheeger inequality]\label{cheei}
We have $\lambda_0(F)\ge h^2/4$.
\end{thm}

In the proofs of \cref{lambdad} and \cref{cheer}, we will need arguments and consequences of the proof of \cref{cheei}
and, therefore, recall the elegant arguments from the proof of the corresponding Lemma 2 in \cite{Os}.

\begin{proof}[Recalling the proof of the Cheeger inequality]
Since $F$ is compact with piecewise smooth boundary,
$\lambda_0=\lambda_0(F)$ is the first Dirichlet eigenvalue of $F$.
Let $\varphi$ be the corresponding ground state 
and set $\psi=\varphi^2$.
By the Schwarz inequality, we have
\begin{equation}\label{chee1}
\begin{split}
	\int_F |\nabla \psi|
	&= \int_F 2 |\varphi| |\nabla \varphi|
	\leq 2 \left(\int_F|\nabla \varphi|^2 \right)^{1/2} \left(\int_F \varphi^2 \right)^{1/2}
	\\
	&= 2 \sqrt{\lambda_0} \int_F \varphi^2
	= 2 \sqrt{\lambda_0} \int_F \psi.
\end{split}
\end{equation}
This implies
\begin{equation}\label{chee2}
  \sqrt{\lambda_0} \ge \frac12\frac{\int_F|\nabla\psi|}{\int_F\psi}.
\end{equation}
For regular values $t>0$ of $\psi$, let $F_t=\{\psi\ge t\}$
and denote by $A(t)$ and $L(t)$ the area and the length of $F_t$
and $\partial F_t=\{\psi=t\}$, respectively.
For the null set of singular values of $\psi$, set $A(t)=L(t)=0$.
The coarea formula gives
\begin{equation}\label{chee3}
  \int_F|\nabla\psi| = \int_0^\infty L(t)dt.
\end{equation}
On the other hand, since $\int_F\psi$ computes the volume of the domain
\begin{equation*}
  \{(x,y)\in F\times\R \mid 0\le y\le \psi(x) \},
\end{equation*}
Cavalieri's principle gives
\begin{equation}\label{chee4}
  \int_F\psi = \int_0^\infty A(t)dt.
\end{equation}
By the definition of $h=h(F)$, we have
\begin{equation}\label{chee5}
  \int_F|\nabla\psi|
  = \int_0^\infty L(t)dt
  \ge h \int_0^\infty A(t)dt = h \int_F\psi.
\end{equation}
Combining \eqref{chee2} and \eqref{chee5},
we get $\lambda_0\ge h^2/4$ as asserted.
\end{proof}

\section{Quantitative estimates of $\Lambda(S)$} \label{secq}

We start with a version of \cref{main2} for surfaces with (possibly empty) boundary.

\begin{thm}\label{mainb}
Let $S$ be a compact and connected Riemannian surface, with or without boundary,
with infinite fundamental group and curvature $K\le\kappa$, where $\kappa$ is a constant.
Then we have: \\
\begin{inparaenum}[1)]
\item\label{mainb1}
If $\kappa\le0$, then
\begin{equation*}
  \Lambda(S)  \ge -\frac{\kappa}4 + \frac1{|S|}\min\left\{\pi,\frac{\sys(S)^2}{|S|}\right\}.
\end{equation*}
\item\label{mainb2}
If $\kappa>0$ and $S$ is orientable, then
\begin{equation*}
  \Lambda(S)  \ge \min\left\{\frac{\pi}{|S|}-\frac{\kappa}4,\frac{\sys(S)^2}{|S|^2}\right\}.
\end{equation*}
\item\label{mainb3}
If $\kappa>0$ and $S$ is non-orientable, then
\begin{equation*}
  \Lambda(S)  \ge \min\left\{\frac{\pi}{|S|}-\frac{\kappa}4,\frac{\sys(S)^2}{4|S|^2}\right\}.
\end{equation*}
\end{inparaenum}
\end{thm}

\begin{proof}
For a closed disc $D$ in $S$, \cref{c1}.\ref{c10} implies that
\begin{equation}\label{ched}
  \frac{|\partial D|^2}{|D|^2} \ge -\kappa + \frac{4\pi}{|D|} \ge -\kappa + \frac{4}{|S|}\pi.
\end{equation}
Suppose now that $A$ is a closed annulus in $S$.
Suppose first that the boundary circles of $A$ are null-homotopic in $S$.
Then by the Schoenflies theorem (see also \cite[Appendix A]{BMM})
there is a disc $D$ in $S\setminus\mathring A$ such that $F'=A\cup D$ is a disc.
Then $|\partial F'|\le |\partial A|$ and $|F'|\ge |A|$.
Using \cref{c1}.\ref{c10} again,
we get that \eqref{ched} 
also holds for $A$ in place of $D$. 

Assume now that the boundary circles of $A$ are not null-homotopic in $S$.
By \cref{c1}.\ref{c1b} and the statement after it, we have
\begin{equation*}
  |\partial A|^2\ge-\min(\kappa,0)|A|^2+4l(A)^2,
\end{equation*}
where $l(A)$ denotes the length of a shortest curve in the free homotopy class
in $A$ of the two boundary circles of $A$.
Since the boundary circles of $A$ are not homotopic to zero in $S$,
we have $l(A)\ge\sys(S)$.
Hence
\begin{equation}\label{chea}
   \frac{|\partial A|^2}{|A|^2} \ge -\min(\kappa,0) + 4\frac{\sys(S)^2}{|A|^2}
   \ge -\min(\kappa,0) + \frac{4}{|S|}\frac{\sys(S)^2}{|S|}.
\end{equation}
If $C$ is a cross cap in $S$, then $S$ is not orientable.
Now the soul of $C$ is not homotopic to zero in $S$
and the fundamental group of $S$ is torsion free.
Since the boundary circle $\partial C$ of $C$ is freely homotopic to the soul of $C$, run twice,
we get that $\partial C$ is not homotopic to zero in $S$.
In particular, we always have $|\partial C|\ge\sys(S)$.
If $\kappa\le0$,
then a shortest curve in $S$ in the free homotopy class of the soul of $C$, run twice,
is a shortest curve in $S$ in the free homotopy class of the boundary circle of $C$.
Hence $|\partial C|\ge2\sys(S)$ if $\kappa\le0$.
We conclude that \eqref{chea} also holds for $C$ in place of $A$ if $\kappa\le0$. 
In the general case, $|\partial C|\ge\sys(S)$ implies a modified version of \eqref{chea}
with $C$ in place of $A$, where the factor $4$ on the right hand side is replaced by $1$.

Now the assertions of \cref{mainb} follows from the Cheeger inequality (\cref{cheei})
in combination with \cref{cheey}, \eqref{ched}, and \eqref{chea} or
the modified version of \eqref{chea}, respectively.
\end{proof}

\begin{proof}[Proof of \cref{main2}]
It remains to show that $\sys(S)^2/|S|\le\pi$ if $S$ is closed
with curvature $K\le0$.
In fact, in that case, the injectivity radius of $S$ is $\sys(S)/2$.
Then the exponential map $\exp_p$ at any point $p\in S$ is a diffeomorphism
from the disc of radius $\sys(S)/2$ in $T_pS$ to its image,
the metric ball $B=B(p,\sys(S)/2)$ about $p$ in $S$.
By comparison with the flat case, we get $|B|\ge\pi\sys(S)^2/4$ and therefore
\begin{equation*}
  \sys(S)^2/|S| < \sys(S)^2/|B| \le 4/\pi.
  \qedhere
\end{equation*}
\end{proof}

\begin{rems}\label{mainc}
\begin{inparaenum}[1)]
\item
If $S$ is a compact and connected surface with non-empty boundary,
then $S$ contains a finite graph $G$ in its interior which is a deformation retract of $S$.
Given a Riemannian metric on S, a sufficiently small tubular neighborhood $T$ of $G$ in $S$
is a Riemannian surface diffeomorphic to $S$ with $\sys(T)\ge\sys(S)$
and with arbitrarily small area.
Moreover, any upper bound on the curvature persists.
In other words, we cannot expect to remove the minimum on the right hand side
of the estimates in \cref{mainb}.

Note also that the right hand side of the inequalities in \ref{mainb2}) and \ref{mainb3})
of \cref{mainb} is positive if and only if $|S|<4\pi/\kappa$,
that is, if and only if $|S|$ is smaller
than the area of the sphere of constant curvature $\kappa>0$.

\item
In \cite[Corollary 5.2.B]{Gr}, Gromov shows that $\sys(S)^2/|S|^2\le 4/3$
for any closed Riemannian surface.
The point is, of course, that his estimate is curvature free.
His work in \cite{Gr} also implies that \[\sys(S_g)^2/|S_g|^2\le C_g(\ln g)^2/g\]
with $\limsup C_g\le 1/\pi$ as $g\to\infty$; see Section 11.3 in \cite{Ka}.
\end{inparaenum}
\end{rems}

\begin{proof}[Proof of \cref{main3} in the case $\chi(S)\ge0$]
In view of \cref{proqua}, it remains to show that $\Lambda(S)>\lambda_0(\tilde S)$
in the case where $S$ is a torus or a Klein bottle.
Then $S$ admits a flat background metric $h$ which is conformal to the given metric $g$ of $S$.
\cref{main2} applies to $h$ and shows that
\begin{equation}\label{lamsys}
  \Lambda(S,h) \ge \sys(S,h)^2/|(S,h)|^2,
\end{equation}
where $(S,h)$ denotes $S$, endowed with the metric $h$.
Furthermore, since we are in the case of surfaces,
the Dirichlet integral of smooth functions is invariant under conformal changes;
that is, we have
\begin{equation}\label{dirich}
   \int |\nabla\vf|^2\,da = \int |\nabla\vf|_h^2\,da_h.
\end{equation}
Since $S$ is compact, there is a constant $\alpha\ge1$ such that
\begin{equation*}
  \alpha^{-1}|v| \le |v|_h\le \alpha |v|
\end{equation*}
for all tangent vectors $v$ of $S$.
Using \eqref{lamsys} and \eqref{dirich},
we obtain
\begin{equation*}
  \Lambda(S) \ge \alpha^{-2}\Lambda(S,h)
  \ge \alpha^{-8}\sys(S)^2/|S|^2 > 0.
\end{equation*}
 On the other hand, the fundamental group of $S$ is amenable
and hence $\lambda_0(\tilde S)=\lambda_0(S)$ by \cite[Theorem 1]{Br2}.
Now $S$ is a torus or a Klein bottle, hence $\lambda_0(\tilde S)=0$.
\end{proof}

\begin{proof}[Proof of \cref{esansy}]
Suppose first that $S$ is orientable, that is, that $S=S_g$ for some $g\ge2$,
and let $c$ be a systolic geodesic on $S$.
Then by \cite[Theorem 4.3.2]{Bu2},
the tubular neighborhood $T$ of $c$ of width
\begin{equation*}
  w_2 = \arsinh(1/\sinh(\sys(S)/2)))
\end{equation*}
is an open annulus.
Since $c$ is essential, $T$ is incompressible.
Note also that $T$ can be exhausted by incompressible compact annuli with smooth boundary.
In particular, for any $r<w_2$,
the closed metric ball $\bar B(p,r)$ of radius $r$ about a point $p$ on $c$
is contained in an incompressible compact annulus $A_r\subseteq T$ with smooth boundary.
Since $B(p,r)\subseteq A_r$,
we may use Theorem 1.1 and the first displayed formula on page 294 of \cite{Ch1}
to conclude that
\begin{equation*}
  \lambda_0(A_r) \le \lambda_0(\bar B(p,r)) \le - \frac{\kappa}{4} + \frac{4\pi^2}{r^2}.
\end{equation*}
By the definition of $\Lambda(S)$,
we have $\Lambda(S) \le \lambda_0(A_r)$ for any $r$ as above.
Hence the claim of \cref{esansy} follows in the case $S=S_g$.

Suppose now that $S$ is not orientable.
Let $\Or(S)\to S$ be the orientation covering of $S$ and $c$ be a systolic geodesic on $S$.
There are two cases:
\begin{inparaenum}[1)]

\item\label{one}
If $c$ is one-sided, then the lift $\tilde c$ of $c$ to $\Or(S)$ is simple of length $2L$
and is invariant under the non-trivial covering transformation $f$ of $\Or(S)$.
Again by \cite[Theorem 4.3.2]{Bu2},
the tubular neighborhood $T$ of $\tilde c$ of width
\begin{equation*}
  w_1 = \arsinh(1/\sinh(\sys(S)))
\end{equation*}
is an open annulus.
Since $f$ leaves $\tilde c$ invariant, it also leaves $T$ invariant
and $T/f$ is an open cross cap with soul $c$ and width $w_1$ about $c$.
Hence for any $r<w_1$, the closed metric ball $\bar B(p,r)$ of radius $r$ about a point $p$ on $c$
is contained in a compact cross cap $C_r\subseteq T/f$ with smooth boundary.

\item\label{two}
If $c$ is two-sided, then $c$ has two lifts $c_1$ and $c_2$ to $\Or(S)$
and both are simple of length $L$.
Moreover, by \cite[Theorem 4.3.2]{Bu2},
the tubular neighborhoods $T_1$ of $c_1$ and $T_2$ of $c_2$
of width $w_2$ are open annuli and do not intersect.
Now $f$ permutes $c_1$ and $c_2$, therefore also $T_1$ and $T_2$,
and hence the tubular neighborhood $T$ of $c$ of width $w_2$ is an open annulus.
\end{inparaenum}

In both cases, \ref{one}) and \ref{two}),
we can now conclude the proof of the claim of \cref{esansy} as in the case $S=S_g$.
\end{proof}

\begin{rem}\label{eslam}
The arguments in the proof of \cref{esansy} also show that $\diam S\ge w$
with $w=w_2$ and $w=w_1$, respectively.
Hence we get
\begin{equation*}
  \lambda_{-\chi(S)} \le -\frac{\kappa}{4} + \chi(S)^2\frac{16\pi^2}{w^2}
\end{equation*}
from Corollary 2.3 of \cite{Ch1}.
In view of $\Lambda(S)<\lambda_{-\chi(S)}$,
this gives another, but weaker upper bound for $\Lambda(S)$.
\end{rem}

\section{On the ground state}\label{secfe}

Throughout this section,
we let $F$ be a compact Riemannian surface with smooth boundary $\partial F\ne\emptyset$
and $\vf$ be the ground state of $F$.
We also set $\psi=\vf^2$ and let $F_t=\{\psi\ge t\}$.
Note that $\int\psi=1$.

By the Hopf boundary lemma \cite[Lemma 3.4]{GT},
$\vf$ does not have critical points on $\partial F$.
Moreover, since $\vf>0$ in the interior $\mathring F$ of $F$,
a point in $\mathring F$ is critical for $\psi$ if and only if it is critical for $\vf$.
All points of $\partial F$ are critical for $\psi$.

In our first result,
we elaborate on the argument from the middle of page 549 in \cite{Os}.

\begin{lem}\label{topft}
Let $0<t<\max\psi$ be a regular value of $\psi$.
Then $F_t$ is a compact subsurface of $\mathring F$
such that the boundary of each component of $F\setminus F_t$
has at least one boundary circle in $\mathring F$
and contains at least one boundary circle of $F$.
\end{lem}

\begin{proof}
Since $t$ is a regular value of $\psi$ with $0<t<\max\psi$,
$F_t$ is a compact subsurface of $\mathring F$ with smooth boundary.
If the boundary of a component $C$ of $F\setminus F_t$ would not contain
a boundary circle of $F$,
then $\varphi$ would be a non-constant superharmonic function on $C$
which attains its maximum $\sqrt t$ along $\partial F_t$, a contradiction.
Clearly, the boundary of $C$ must have at least one boundary circle
in $\partial F_t\subseteq\mathring F$.
\end{proof}

\begin{prop}\label{topft2}
Let $0<t<\max\psi$ be a regular value of $\psi$.
Then any connected component $C$ of $F_t$ is incompressible in $F$
and no component of $F\setminus C$ is a disc or a cross cap.
Furthermore, $\chi(C)\ge\chi(F)$ with equality if and only
if $F\setminus C$ is a collared neighborhood of $\partial F$,
consisting of annuli about the boundary circles of $F$.
\end{prop}

\begin{proof}
Suppose that a component $D$ of $F\setminus C$ would be a disc or a cross cap.
Then $D\setminus F_t$ would consist of components of $F\setminus F_t$
with boundary in $\mathring F$, a contradiction to \cref{topft}.
Substituting $C$ for $F'$,
the rest of the proof of \cref{topft2} is now more or less the same as that of \cref{cheey}.
\end{proof}

As in \cref{l1}, we denote by $\ell$ the sum of the lengths of the shortest loops
in the free homotopy classes (in $F$) of the boundary circles of $F$.
Furthermore, we let $\Lambda'(F)=\inf\lambda_0(F')$,
where the infimum is taken over all incompressible
compact and connected subsurfaces $F'$ of $\mathring F$
with smooth boundary and Euler characteristic $\chi(F')>\chi(F)$.

\begin{lem}\label{lambdad}
If $\chi(F)\le0$, then
\begin{equation*}
  \lambda_0(F)\ge \big\{1-\delta
  + 2\big(1-\frac1\delta\big)\frac{|F|}{\ell}\sqrt{\lambda_0(F)}\big\}\Lambda'(F)
\end{equation*}
for all $0<\delta<1/2$.
\end{lem}

\begin{proof}
Since the quantities involved in the lemma vary continuously with respect to variations
of the metric (in the $C^0$-topology),
we may assume, by Theorem 8 in \cite{Uh}, that $\varphi$ is a Morse function.
Then the critical points of $\psi$ in $\mathring F$ are non-degenerate.
Moreover, since $\vf$ does not have critical points on $\partial F$,
$F\setminus F_t$ is a collared neighborhood of $\partial F$,
consisting of annuli about the boundary circles of $F$, for all sufficiently small $t>0$.
On the other hand, for $t<\max\psi$ suffciently close to $\max\psi$,
$F_t$ is a union of embedded discs, one for each maximum point of $\psi$.
Hence the topology of $F_t$ undergoes changes as $t$ increases from $0$ to $\max\psi$.

Since $\vf$ is a Morse function, $\psi$ has only finitely many critical points in $\mathring F$.
By \cref{topft}, $\psi$ does not have local minima in $\mathring F$.
Hence critical points of $\psi$ in $\mathring F$ are saddle points and local maxima.

Let $0=\beta_0<\dots<\beta_m=\max\psi$ be the finite sequence of critical values of $\psi$
and choose $\ve>0$ with $\ve<\min\{\beta_{i+1}-\beta_i\}$.
In a first step,
we select now a critical value $\beta=\beta_i$ according to specific requirements.

By \cref{topft2},
each component $C$ of $F_{\beta_1+\ve}$ has Euler characteristic $\chi(C)\ge\chi(F)$.
Therefore there are two cases.
Either each component $C$ of $F_{\beta_1+\ve}$ has Euler characteristic $\chi(C)>\chi(F)$.
Then we set $\beta=\beta_1$.
Or else there is a component $C$ with $\chi(C)=\chi(F)$.
Then $F\setminus C$ is a collared neighborhood of $\partial F$,
consisting of annuli about the boundary circles of $F$, by \cref{topft2}.
In that case, by \cref{topft}, the other components of $F_{\beta_1+\ve}$ are discs
contained in these annuli.

We assume that we are in the second case and consider the second critical value $\beta_2$.
By \cref{topft2}, there are again two cases.
Either each component $C$ of $F_{\beta_2+\ve}$ has Euler characteristic $\chi(C)>\chi(F)$;
then we set $\beta=\beta_2$.
Or else there is a component $C$ of $F_{\beta_2+\ve}$  with $\chi(C)=\chi(F)$.
Then $F\setminus C$ is a collared neighborhood of $\partial F$
consisting of annuli about the boundary circles of $F$.
In the latter case, we pass on to the next critical value $\beta_3$.
Since $\chi(F)\le0$,
we will eventually arrive at a first critical value $\beta=\beta_i$ with the property
that the complement of a component of $F_{\beta-\ve}$ is a collared neighborhood of $\partial F$
consisting of annuli about the boundary circles of $F$
and such that each component $C$ of $F_{\beta+\ve}$ has Euler characteristic $\chi(C)>\chi(F)$.
Note that this property then holds for all sufficiently small $\ve>0$
since $\beta$ is the only critical value of $\vf$ in $(\beta_{i-1},\beta_{i+1})$.
It follows that for any regular value $0<t<\beta$ of $\psi$,
$F_t$ has a component $C$ such that $F\setminus C$
is a collared neighborhood of $\partial F$
consisting of annuli about the boundary circles of $F$.
In particular, $|\partial F_t|\ge \ell$ for all such $t$.
Using \eqref{chee2} and \eqref{chee3}, we obtain
\begin{equation}\label{cheeta5}
  \beta \ell \le \int_0^\infty L(t) dt \le 2\sqrt{\lambda_0(F)},
\end{equation}
where $L(t)$ denotes the length of $\partial F_t$.

For $\ve > 0$ as above,
the smooth function $\vf_\ve=\varphi -\sqrt{\beta + \ve}$ is smooth on $F_{\beta+\ve},$
vanishes on $\partial F_{\beta+\ve}$ and satisfies
\begin{equation}\label{terms}
\begin{split}
  \int_{F_{\beta+\ve}}\vf_\ve^2
  &= \int_{F_{\beta+\ve}}\big(\varphi - \sqrt{\beta + \ve}\big)^2 \\
  &= \int_{F_{\beta+\ve}}\varphi^2 - 2 \sqrt{\beta + \ve}
  \int_{F_{\beta+\ve}}\varphi + (\beta + \ve) |F_{\beta+\ve}|.
\end{split}
\end{equation}
Now the first term on the right hand side of \eqref{terms} satisfies
\begin{equation}\label{terms2}
  \int_{F_{\beta+\ve}}\varphi^2 \ge \int_F\vf^2 -(\beta+\ve)(|F|-|F_{\beta+\ve}|)
\end{equation}
since $\vf^2\le\beta+\ve$ on $F\setminus F_{\beta+\ve}$.
For the second term on the right hand side of \eqref{terms}, we have
\begin{equation}\label{terms4}
\begin{split}
  2 \sqrt{\beta + \ve} \int_{F_{\beta+\ve}}\varphi
  &\le 2\sqrt{\beta+\ve}|F_{\beta+\ve}|^{1/2}\big(\int_{F_{\beta+\ve}}\vf^2\big)^{1/2} \\
  &\le \frac1\delta(\beta+\ve)|F_{\beta+\ve}| + \delta \int_{F_{\beta+\ve}}\varphi^2 \\
  &\le \frac1\delta(\beta+\ve)|F_{\beta+\ve}| + \delta \int_{F}\varphi^2
\end{split}
\end{equation}
by the Schwarz inequality and the Peter and Paul principle.
Combining \eqref{terms}, \eqref{terms2}, and \eqref{terms4}
and using that the $L^2$-norm of $\vf$ is one and that $2-1/\delta<0$,
we obtain
\begin{align*}
  \int_{F_{\beta+\ve}}\vf_\ve^2 \notag
  &\ge (1-\delta) \int_F\varphi^2 -(\beta+\ve)|F| + \big(2-\frac1\delta\big)(\beta+\ve)|F_{\beta+\ve}| \\
  &\ge 1-\delta + \big(1-\frac1\delta\big)(\beta+\ve)|F|.
\end{align*}
For the Rayleigh quotient of $\vf_\ve$, we get
\begin{align*}
  &R(\vf_\ve)(1-\delta + \big(1-\frac1\delta\big)(\beta+\ve)|F|)
  \le R(\vf_\ve)\int_{F_{\beta+\ve}}\vf_\ve^2 \\
  &\hspace{19mm}
  = \int_{F_{\beta+\ve}} |\nabla \vf_\ve|^2
  = \int_{F_{\beta+\ve}} |\nabla \varphi|^2
  \le \int_F|\nabla \varphi|^2
  = \lambda_0(F).
\end{align*}
Since $F_{\beta + \ve}$ is a disjoint union of incompressible compact and connected
subsurfaces $F'$ with smooth boundary and $\chi(F')>\chi(F)$,
we also have $R(\vf_\ve) \ge \inf\Lambda'(F)$.
Letting $\ve$ tend to $0$, we finally obtain
\begin{equation}\label{lambdabe}
  \Lambda'(F)(1-\delta + \big(1-\frac1\delta\big)\beta|F|)
  \le \lambda_0(F).
\end{equation}
Combining \eqref{cheeta5} and \eqref{lambdabe}, we arrive at \cref{lambdad}.
\end{proof}


\section{On the ground state (continued)}\label{secfe2}

Let $S$ be a complete and connected Riemannian surface of finite type
with $\chi(S)<0$ and $\lambda_0(\tilde S)<\lambda_{\rm ess}(S)$.

Since $\chi(S)<0$, $S$ carries complete hyperbolic metrics.
Using a decomposition of $S$ into a finite number of pairs of pants,
it is clear that we may choose such a metric $h$ such that the connected
components of a neighborhood of the ends of $S$ is a finite union
of \emph{hyperbolic funnels}, that is, cylinders of the form $(-1,\infty)\times\R/\Z$
with metric
\begin{equation*}
  dr^2 + \cosh(r)^2d\vartheta^2.
\end{equation*}
Then the curves $r=0$ are closed $h$-geodesics of length $1$.
The original metric of $S$ will be denoted by $g$.

We fix a smooth and proper function from $S$ to $[0,\infty),$ which agrees outside a compact set with the coordinates $r$ in each of the ends.
By abuse of notation, we denote this function by $r.$
Choose an increasing sequence $0<r_0< r_1< r_2<\dots\to\infty$.
Then the subsurfaces
\begin{equation}\label{exhau}
   K_i = \{r\le r_i\}
\end{equation}
of $S$ are compact with smooth boundary $\partial K_i=\{r=r_i\}$
such that $S\setminus K_i$ is a cylindrical neighbourhood of infinity.
Furthermore,
\begin{equation}\label{exhau2}
   K_0 \subseteq K_1 \subseteq K_2 \subseteq \dots
\end{equation}
is an exhaustion of $S$.
By choosing the sequence of $r_i$ suitably, we may assume that \\
\begin{inparaenum}[a)]
\item
there exist cutoff functions $\eta_i\colon S\to [0,1]$ with $\eta_i=1$ on $K_{i}$,
$\eta_i=0$ outside $K_{i+1}$, and $|\nabla \eta_i|^2 \le 1/i$, \\
\item
$\lambda_0(\tilde S) < \lambda_0(S\setminus K_0)$, \\
\end{inparaenum}
where we note that
$\lambda_0(S\setminus K_i)<\lambda_0(S\setminus K_{i+1})\cdots\to\lambda_{\rm ess}(S)$.

In the case where $S$ is compact, we have $K_i=S$ for all $i$
and part of the following discussion becomes trivial.

We now let $F$ be a compact subsurface of $S$
with smooth boundary $\partial F\ne\emptyset$.
As in \cref{secfe}, we denote by $\vf$ the ground state of $F$ and let $F_t=\{\vf^2\ge t\}$.

\begin{lem}\label{vfreg}
For a subset $\mathcal R\subseteq(0,\max\vf^2)$ of full measure,
$F_t$ is a smooth subsurface of $\mathring F$
such that $\partial F_t=\{\vf=t\}$ and $\partial K_i$ intersect transversally for all $i$.
\end{lem}

\begin{proof}
Since $\vf$ is smooth up to the boundary of $F$ and has no critical points on $\partial F$,
there is a smooth extension $\tilde\vf$ of $\vf$ to $S$
such that $\tilde\vf$ is strictly negative on $S\setminus F$.
The restriction $\tilde\vf_0$ of $\tilde\vf$ to the union of the curves $\{r=r_i\}$ is then smooth,
and hence there is a set $\mathcal R_0\subseteq\R$ of full measure
such that any $t\in\mathcal R_0$ is a regular value of $\tilde\vf_0$.
Note that $\nabla\tilde\vf$ is not perpendicular to the curve $\{r=r_i\}$
at points $p\in\{r=r_i\}$ with $\tilde\vf(p)\in\mathcal R_0$.
On the other hand, $\nabla\tilde\vf$ is perpendicular to $\partial F_t$ for any
regular value $t$ of $\vf^2$ in $(0,\max\vf^2)$.
Therefore the intersection $\mathcal R$ of $\mathcal R_0$
with the set of regular values of $\vf^2$ in $(0,\max\vf^2)$ satisfies the required assertions.
\end{proof}

\begin{lem} \label{topft3}
For any $t\in\mathcal R$, the intersection $F_t \cap K_i$ is a subsurface of $F$
with piecewise smooth boundary
and any connected component $C$ of $F_t \cap K_i$ is 
incompressible in $F$.
In particular, we have $\chi(C)\ge\chi(F)$.
\end{lem}

\begin{proof}
For any $t\in\mathcal R$,
$\partial F_t=\{\vf=t\}$ and $\partial K_i$ intersect transversally for all $i$,
and then $F_t\cap K_i$ is a subsurface of $F$ with piecewise smooth boundary.
Since $S\setminus K_i$ is a cylindrical neighborhood of the ends of $S$,
a disc in $S$ has to be contained in $K_i$ if its boundary is in $K_i$.
Hence the components of $F_t \cap K_i$ are incompressible in $F_t$.
By \cref{topft2}, $F_t$ is incompressible in $F$.
Therefore $F_t \cap K_i$ is incompressible in $F$.
\end{proof}

\begin{lem} \label{comploc}
Assume that $\lambda_0(F)\le\theta\lambda_0(S\setminus K_0)$
for some $0<\theta<1$, and let $\ve>0$.
Then there is an integer $i_0=i_0(\theta,\ve)\ge0$ such that
\begin{equation*}
  \int_{F\cap K_i} \vf^2 \ge 1-\ve
  \quad\text{for all $i\ge i_0$.}
\end{equation*}
\end{lem}

\begin{proof}
Since $(1-\eta_i) \vf$ has support in $F\setminus K_i\subseteq S\setminus K_0$, we have
\begin{align*}
		\lambda_0(S\setminus K_0) \int_F (1-\eta_i)^2 & \vf^2
		\leq \int_F |\nabla ( (1-\eta_i)\vf)|^2
		\\
		&= \int_F \nabla ((1-\eta_i)^2\vf) \cdot \nabla \vf + \int_F \vf^2 |\nabla(1-\eta_i)|^2
		\\	
		&= \int_F ((1-\eta_i)^2\vf) \cdot \Delta\vf + \int_F \vf^2 |\nabla(1-\eta_i)|^2
		\\	
		&\leq \lambda_0(F) \int_F  (1-\eta_i)^2 \vf^2 + \frac{1}{i} \int_F \vf^2
		\\
	       &\leq \theta \lambda_0(S\setminus K_0) \int_F(1-\eta_i)^2 \vf^2 + \frac{1}{i}.
\end{align*}
Since $0<\theta<1$, we conclude that
\begin{equation*}
		\int_{F} (1-\eta_i)^2 \vf^2 \leq \frac{1}{(1-\theta)\lambda_0(S\setminus K_0)i}.
\end{equation*}
Now for $i_0$ sufficiently large,
the right hand side is smaller than $\ve$ for all $i \geq i_0-1$.
For any $i\ge i_0$, we then have
\begin{equation}
	\begin{split}
		\int_{F\cap K_{i}} \vf^2 &= 1-\int_{F\setminus K_{i}} \vf^2
		\\
		&= 1- \int_{F\setminus K_{i}} (1-\eta_{i-1})^2 \vf^2	
		\\
		&\geq 1 - \int_F (1-\eta_{i-1})^2 \vf^2
		\\
		&\geq 1- \ve.	\qedhere
	\end{split}
\end{equation}
\end{proof}

There is a sequence of constants $1\le\alpha_0\le\alpha_1\le\cdots$
such that
\begin{equation}\label{normc}
  \alpha_i^{-1}|v| \le |v|_h\le \alpha_i|v|
\end{equation}
for all tangent vectors $v$ of $S$ with foot point in $K_i$,
where no index and index $h$ indicate measurement with respect to $g$ and $h$,
respectively.
Over $K_i$,
the area elements $da$ of $g$ and $da_h$ of $h$ are then estimated by
\begin{equation}\label{normc2}
  \alpha_i^{-2}da \le da_h \le \alpha_i^2 da
\end{equation}
with corresponding inequalities for the areas of measurable subsets
and for integrals of non-negative measurable functions.

Let now again $\varphi$ be the ground state of $F$,
$t\in\mathcal R$, and $F_t=\{\vf^2\ge t\}$.
In our next result,
we estimate the inradius of $F_t \cap K_i$ for sufficiently large $i$.

\begin{lem}\label{cheer}
Let $F$ be a disc, an annulus, or a cross cap.
Assume that $\lambda_0(F)\le\theta\lambda_0(S\setminus K_0)$
for some $0<\theta<1$ and let $\delta>0$.
Then there is an integer $i_1=i_1(\theta,\delta)\ge0$ such that
the inradius $\rho(t)$ of $F_t \cap K_i$ satifies
\begin{equation*}
  \coth(\alpha_{i+1}\rho(t))
  \le \frac{2\alpha_{i+1}^3\sqrt{\lambda_0(F)+ \delta}}{1- \delta -t|F \cap K_i|}
\end{equation*}
for all $0\le t<(1-\delta)/|F \cap K_i|$ and $i\ge i_1$.
\end{lem}

\begin{proof}
In a first step, we estimate the Rayleigh quotient of $\eta_i\vf$.
Computing as in the proof of \cref{comploc}, we have
\begin{align*}
  \int_F |\nabla(\eta_i\vf)|^2
  &= \int_F \nabla(\eta_i^2\vf)\nabla\vf + \int_F|\nabla\eta_i|^2\vf^2 \\
  &= \int_F \eta_i^2\vf\Delta\vf + \int_F|\nabla\eta_i|^2\vf^2 \\
  &\leq \lambda_0(F) \int_F \eta_i^2\vf^2 + 1/i.
\end{align*}
Since $\eta_i=1$ on $K_i$, we get $R(\eta_i \vf) \le \lambda_0(F) + 2/i$
for all $i\ge i_0(\theta,1/2)$, where $i_0$ is taken from \cref{comploc}.
Therefore
\begin{equation} \label{coray}
  R(\eta_i \vf) \le \lambda_0(F) + \delta
  \quad\text{for all $i \ge i_1(\theta,\delta)$,}
\end{equation}
where we may assume that $i_1(\theta,\delta)\geq i_0(\theta,\delta).$
In a second step, we follow the proof of Cheeger's inequality, \cref{cheei}.
Computing as in \eqref{chee1}, we get
\begin{equation*}
  \int_F|\nabla (\eta_i^2 \vf^2)|
  \le 2\sqrt{R(\eta_i\vf)}\int_F \eta_i^2 \vf^2
  \le 2\sqrt{R(\eta_i\vf)}.
\end{equation*}
By the coarea formula and \eqref{normc}
and since $\supp\eta_i\subseteq K_{i+1}$ and $\eta_i=1$ on $K_{i}$,
we have
\begin{equation}\label{cheeta}
	\begin{split}
  \alpha_{i+1} \int_F|\nabla (\eta_i^2\vf^2)|
 &= \alpha_{i+1} \int_0^\infty |\{\eta_i^2\vf^2=s\}|\,ds \\
 &\ge \int_t^\infty |\{\eta_i^2\vf^2=s\}|_h\,ds \\
 &= \int_t^\infty |\{\vf^2=s\}\cap K_i|_h\,ds \\
 &\hspace{9mm}
 + \int_t^\infty |\{\eta_i^2\vf^2=s\}\cap(F\setminus K_i)|_h\,ds
  	\end{split}
\end{equation}
Here we note that, for the integration, it suffices to consider $s\in\mathcal R$
which are also regular for $\eta_i^2\vf^2$.
Then $\{\vf^2=s\}$ meets $\partial K_i$ transversally and,
therefore, $\{\eta_i^2\vf^2=s\}\cap(F\setminus K_i)$ consists
of arcs $a_j$ connecting their corresponding end points
on $\{\vf^2=s\}\cap\partial K_i$.
Replacing the $a_j$ by the corresponding segments $b_j$ on $\partial K_i$,
we obtain the boundary of the subsurface $F_s\cap K_i$.
Now $|b_j|_h\le|a_j|_h$ by the choice of the hyperbolic metric $h$ on $S$
and since the numbers $r_i$ defining the $K_i$ are positive.
Hence we have
\begin{equation}
\begin{split}
|\{\vf^2=s\}\cap K_i|_h  & +  |\{\eta_i^2\vf^2=s\} \cap F\setminus K_i|_h
\\
& = |\{\vf^2=s\}\cap K_i|_h + \sum |a_j|_h
\\
& \geq |\{\vf^2=s\}\cap K_i|_h + \sum |b_j|_h
\\
& = |\partial (\{\vf^2 \geq s\}\cap K_i)|_h,
\end{split}
\end{equation}
for any $s \in \mathcal{R}.$
This implies
\begin{equation} \label{cheeta2}
\begin{split}
  \int_t^\infty |\{\vf^2=s\}&\cap K_i|_h\,ds
    + \int_t^\infty |\{\eta_i^2\vf^2=s\}\cap(F\setminus K_i)|_h\,ds
  \\
   & \geq \int_t^\infty |\partial(\{\vf^2\ge s\}\cap K_i)|_h\,ds.
\end{split}
\end{equation}
Furthermore, by \cref{topft3},
the connected components of $F_s\cap K_i$
are subsurfaces of $F$ with piecewise smooth boundary
and have non-negative Euler characteristic for any $s\in\mathcal R$.
Therefore we get
\begin{equation}\label{inroh}
\begin{split}
  \int_t^\infty |\partial(\{\vf^2\ge s\}&\cap K_i)|_h\,ds \\
  &\ge \int_t^\infty |\{\vf^2\geq s\} \cap K_i|_h\coth(\rho_h(s))\,ds \\
  &\ge \coth(\rho_h(t)) \int_t^\infty |\{\vf^2 \geq s\} \cap K_i|_h\,ds \\
  &\ge \alpha_i^{-2}\coth(\alpha_i\rho(t)) \int_t^\infty |\{\vf^2 \geq s\} \cap K_i|\,ds,
\end{split}
\end{equation}
by \cref{c1}.\ref{c1a}, \eqref{normc}, and \eqref{normc2}.
Finally, since $i_1(\theta,\delta) \geq i_0(\theta,\delta)$,
\begin{equation}\label{cheetb}
\begin{split}
  \int_t^\infty |\{\vf^2 \geq s\} &\cap K_i|\,ds
  \ge \int_0^\infty |\{\vf^2 \geq s\} \cap K_i|\,ds - t|F\cap K_i| \\
  &= \int_{F\cap K_i}\vf^2 - t|F\cap K_i|
  \geq 1- \delta -t|F\cap K_i|.
\end{split}
\end{equation}
\cref{cheer} follows now from combining \eqref{coray} -- \eqref{cheetb}.
\end{proof}

Lemmas \ref{lambdad} and \ref{cheer} will lead to the apriori estimates in \cref{cheet}
and in the proof of \cref{isotyp},
which are essential in the proof of \cref{main1}.


\section{Qualititative estimates of $\Lambda(S)$} \label{sece}

In this section, we prove \cref{main3} in the case $\chi(S)<0$.
Throughout,
we let $S$ be a complete and connected Riemannian surface of finite type and set
\begin{equation}\label{lamdac}
  \Lambda_D(S)= \inf_D \lambda_0(D),
  \hspace{3mm}
  \Lambda_A(S)= \inf_D \lambda_0(A),
  \hspace{3mm}
  \Lambda_C(S)= \inf_D \lambda_0(C),
\end{equation}
where the infimum is taken over all embedded closed discs $D$,
incompressible annuli $A$, and cross caps $C$ in $\mathring S$ with smooth boundary,
respectively.
As we will explain in \cref{sused}, we have
\begin{equation*}
  \Lambda_D(S)\ge\Lambda_A(S)
  \quad\text{and}\quad
  \Lambda(S) = \inf\{\Lambda_A(S),\Lambda_C(S)\}
\end{equation*}
if the fundamental group of $S$ is infinite.
Nevertheless, since the case of discs reveals an essential idea of the proof
and since we will need the estimate anyway,
we include the discussion of $\Lambda_D(S)$.

We fix a hyperbolic metric $h$ on $S$ as in \cref{secfe2}
and denote by $g$ the original Riemannian metric of $S$.
If not otherwise mentioned, statements refer to $g$ and not to $h$.

We will use the setup and notation from the previous section.
The following assertion is an immediate consequence of \cref{cheer}.

\begin{lem}\label{cheet}
Let $F$ be a closed disc, annulus, or cross cap in $S$
and $\varphi$ be the ground state of $F$.
Assume that $\lambda_0(F)\le\theta\lambda_0(S\setminus K_0)$ for some $0<\theta<1$.
Then the inradius $\rho(\ve)$ of $\{\vf^2\ge\ve\} \cap K_i$
satisfies $\rho(\ve)\ge\rho>0$ for all $0<\ve<1/4|K_{i_1}|$ and $i\ge i_1=i_1(\theta,1/2)$.
\end{lem}

We now discuss the cases of discs and annuli separately.

\begin{thm} \label{discs}
If $S$ is a complete and connected Riemannian surface of finite
type with $\chi(S)<0$ and $\lambda_{\ess}(S)>\lambda_0(\tilde S)$,
then $\Lambda_D(S)>\lambda_0(\tilde{S})$.
\end{thm}

\begin{proof}
Suppose that there is a sequence of discs $D_n$ in $S$
with smooth boundary such that $\lambda_0(D_n)\to\lambda_0(\tilde S)$.
Let $\varphi_n\colon D_n\to\R$ be the positive $\lambda_0(D_n)$-Dirichlet eigenfunction
with $||\varphi_n||_2=1$.
By passing to a subsequence if necessary,
we may assume that $\lambda_0(D_n)\le\theta \lambda_0(S\setminus K_0)$
for some $0<\theta<1$.
By \cref{cheet} and up to passing to a subsequence,
there are positive constants $\ve_0$ and $\rho_0$
and a point $x_0\in S$ such that $B(x_0,2\rho_0)$
is contained in $\{\varphi_n^2\ge\ve_0\} \cap K$ for all $n$,
where $K=K_{i_1(\theta,1/2)}$.

Fix a point $\tilde x_0\in\tilde S$ above $x_0$.
Then there is a unique lift $\tilde D_n$ of $D_n$ to $\tilde S$ containing $\tilde x_0$
such that $\tilde D_n\to D_n$ is a diffeomorphism (including the boundary).
Thus we may also lift $\vf_n$ to $\tilde\vf_n$ on $\tilde D_n$ and extend $\tilde\vf_n$
to a function $\tilde\vf_n$ on $\tilde S_n$
by setting $\tilde\vf_n=0$ on $\tilde S\setminus\tilde D_n$.
Since the boundary of $\tilde D_n$ is smooth and $\tilde\vf_n$ is smooth on $\tilde D_n$,
it follows that $\tilde\vf_n\in H^1_0(\tilde S)$ with $H^1$-norm
\begin{equation*}
  ||\tilde\vf_n||_{H^1} = ||\varphi_n||_{H^1}^2 = \lambda_0(D_n) + 1,
\end{equation*}
where we use Green's formula for the second equality.
In particular, up to extracting a subsequence, we have weak convergence
\begin{equation*}
  \tilde\vf_n \rightharpoonup \tilde\vf\in H^1_0(\tilde S)
  \quad\text{with}\quad ||\tilde\vf||_{H^1} \le \liminf ||\varphi_n||_{H^1}.
\end{equation*}
Up to extracting a further subsequence,
the sequence of $\tilde\vf_n$ converges uniformly in any $C^k$-norm in $B(\tilde x_0,\rho_0)$,
by Theorem 8.10 in \cite{GT}.
In particular $\tilde\varphi^2\ge\ve_0$ on $B(\tilde x_0,\rho_0)$.

By Theorem 1 of \cite{AFLMR},
we may approximate the distance function $d_0$ to $\tilde x_0$ in $\tilde S$
by a smooth function $u$ on $\tilde S$ such that $|u-d_0|\le1$ and $|\nabla u|\le2$.
Then the sublevels $B(r)=\{u\le r\}$ form an exhaustion of $\tilde S$ by compact subsets.
Clearly,
\begin{equation*}
  R(\tilde\vf) = \lim_{r\to\infty} R(\tilde\vf|_{B(r)}).
\end{equation*}
Furthermore, up to passing to a subsequence, we have weak convergence
\begin{equation*}
  \tilde\vf_n|_{B(r)}\rightharpoonup\tilde\vf|_{B(r)} \quad\text{in $H^1(B(r))$}
\end{equation*}
and strong convergence
\begin{equation*}
  \tilde\vf_n|_{B(r)}\to\tilde\vf|_{B(r)} \quad\text{in $L^2(B(r))$.}
\end{equation*}
Hence
\begin{equation*}
  R(\tilde\vf|_{B(r)})\le\liminf R(\tilde\vf_n|_{B(r)}).
\end{equation*}
For any regular value $r$ of $u$ such that $\partial B(r)$ intersects $\partial \tilde D_n$ transversally we have
\begin{equation*}
\begin{split}
  \int_{B(r)}|\nabla\tilde\vf_n|^2
  &= \lambda_0(D_n) \int_{B(r)} |\tilde\vf_n|^2
  + \int_{\partial B(r)}\tilde\vf_n\la\nabla\tilde\vf_n,\nu\ra,
\end{split}
\end{equation*}
where $\nu=\nabla u/|\nabla u|$ is the outward unit vector field along $\partial B(r)$.
Clearly, the second term on the right satisfies
\begin{equation*}
  \int_{\partial B(r)}\tilde\vf_n\la\nabla\tilde\vf_n,\nu\ra
  \le \int_{\partial B(r)} (|\tilde\vf_n|^2 + |\nabla\tilde\vf_n|^2).
\end{equation*}
Let now $\ve>0$ be given.
Then there is a sequence of integers $k_m\to\infty$ such that,
for each $m$, there is a subsequence of $n\to\infty$ with
\begin{equation} \label{subseq}
  \int_{K(k_m)} (|\tilde\vf_n|^2 + |\nabla\tilde\vf_n|^2) < \ve,
\end{equation}
where $K(r)=B(r)\setminus B(r-1)$.
If this would not be the case, there would be some $\ve>0$ and a positive integer $m$
such that for all integers $k \geq m$ there is an integer $l$ such that, for all integers $n \geq l$,
\begin{equation} \label{contra}
\int_{K(k)} (|\tilde \varphi_n|^2 + |\nabla \tilde \varphi_n|^2) \geq \ve.
\end{equation}
We thus find, for any fixed $M$, an index $N$ such that \eqref{contra} holds
for $k=m, \dots, m+M$ and $n \geq N.$
If we choose $M$ such that $M \ve > \lambda_0(\tilde S)+1,$ we get a contradiction.

By the coarea formula, we have
\begin{equation*}
  \int_{K(k_m)} (|\tilde \varphi_n|^2 + |\nabla \tilde \varphi_n|^2)|\nabla u|
  = \int_{k_m-1}^{k_m} \int_{\partial B(r)} (|\tilde \varphi_n|^2 + |\nabla \tilde \varphi_n|^2).
\end{equation*}
Since $|\nabla u|\le2$, we then get, for $n$ as in \eqref{subseq},
that there is a regular value $r_n\in (k_m-1,k_m)$ of $u$ such that
\begin{equation*}
   \int_{\partial B(r_n)} (|\tilde\vf_n|^2 + |\nabla\tilde\vf_n|^2) < 2\ve,
\end{equation*}
where we may also assume that $\partial B(r_n)$ intersects $\partial \tilde D_n$ transversally.
We obtain
\begin{equation*}
\begin{split}
  R(\tilde\vf_n|_{B(k_m)})
  &\le \frac{\int_{B(r_n)} |\nabla\tilde\vf_n|^2}{\int_{B(k_m)} |\tilde\vf_n|^2}
  + \frac{\int_{K(k_m)} |\nabla\tilde\vf_n|^2}{\int_{B(k_m)}|\tilde\vf_n|^2}  \\
  &\le \frac{\int_{B(r_n)} |\nabla\tilde\vf_n|^2}{\int_{B(r_n)} |\tilde\vf_n|^2}
  + \frac{\ve}{\int_{B(k_m)}|\tilde\vf_n|^2}  \\
  &\le  \lambda_0(D_n) + \frac{
  \int_{\partial B(r_n)} \tilde\vf_n\la\nabla\tilde\vf_n,\nu\ra}{\int_{B(r_n)} |\tilde\vf_n|^2}
  + \frac{\ve}{\int_{B(k_m)}|\tilde\vf_n|^2}\\
  &\le \lambda_0(D_n) + \frac{2\ve}{\int_{B(r_n)}|\tilde\vf_n|^2}
  + \frac{\ve}{\int_{B(k_m)}|\tilde\vf_n|^2}.
\end{split}
\end{equation*}
It follows that
\begin{equation*}
  R(\tilde\vf|_{B(k_m)})
  \le \lambda_0(\tilde S) + \frac{3\ve}{\int_{B(\tilde x_0,\rho_0)}|\tilde\vf|^2}
\end{equation*}
for all sufficiently large $m$.
In conclusion,
\begin{equation*}
  R(\tilde\vf) \le \lambda_0(\tilde S).
\end{equation*}
Since $\tilde\vf\in H^1_0(\tilde S)$, this implies that $\tilde\vf$ is an eigenfunction
of the Laplacian with eigenvalue $\lambda_0(\tilde S)$.

Now $\tilde\vf$ is non-zero on $B(\tilde x_0,\rho_0)$.
On the other hand, by the definition of the lifts $\tilde\vf_n$,
$\tilde\vf$ vanishes on any other preimages $B(\tilde x,\rho_0)$ of $B(x_0,\rho_0)$
under the covering projection $\tilde S\to S$.
Now the fundamental group of $S$ is not trivial, hence there are such preimages $\tilde x\in\tilde S$.
Thus we arrive at a contradiction to the unique continuation property
for eigenfunctions of Laplacians \cite{Ar}.
\end{proof}

\begin{thm} \label{annuli}
If $S$ is a complete and connected Riemannian surface of finite
type with $\chi(S)<0$ and $\lambda_{\ess}(S)>\lambda_0(\tilde S)$,
then $\Lambda_A(S)>\lambda_0(\tilde{S})$.
\end{thm}

To prove \cref{annuli},
we assume the contrary and let $(A_n)$ be a sequence of incompressible annuli in $S$
with smooth boundary such that $\lambda_0(A_n)\to\lambda_0(\tilde S)$.
We may assume that
\begin{equation}\label{pre2}
  \lambda_0(A_n) + 4\delta < \min(\Lambda_D(S),\theta\lambda_0(S\setminus K_0))
\end{equation}
for all $n$ and some fixed constants $\delta,\theta\in(0,1)$,
by invoking \cref{discs} and that
$\lambda_0(\tilde S)<\lambda_0(S\setminus K_0)$
by the choice of $K_0$ in \eqref{exhau2}.
By deforming the $A_n$ (slightly), we may also assume that
\begin{equation}\label{pre4}
  \text{$\partial A_n$ and $\partial K_i$ intersect transversally}
\end{equation}
for all $n$ and $i$.
Then the intersections $A_n\cap K_i$ and $A_n\cap(S\setminus\mathring K_i)$
are incompressible subsurfaces of $S$ with piecewise smooth boundary.

Recall the constant $ i_1(\theta,\delta)$ from \cref{cheer}.

\begin{lem}\label{isotyp}
By passing to a subsequence, we may assume that \\
\begin{inparaenum}[1)]
\item\label{p2}
all $A_n$ are isotopic in $S$ and, for each $i>i_1=i_1(\theta,\delta)$,
exactly one component $A_n'$ of $A_n\cap K_i$ is an annulus
(topologically) isotopic to $A_n$; \\
\item\label{p4}
there is a constant $\ell_0>0$ such that the free homotopy classes
of the boundary curves of $A_n$ in $A_n$ contain curves of length at most $\ell_0$
with respect to $g$ and $h$; \\
\item\label{p6}
there are $x_0\in S$ and $\rho,\ve>0$ such that $B(x_0,\rho)\subseteq A_n$
and such that the ground states $\vf_n$ of $A_n$
satisfy $\varphi_n\ge\ve$ on $B(x_0,\rho)$.
\end{inparaenum}
\end{lem}

\begin{proof}
The connected components of $A_n\cap\partial K_i$ consist of
embedded segments connecting two boundary points of $A_n$
and of embedded circles in the interior of $A_n$.
By the Schoenflies theorem and the topology of $S$,
there are the following two possibilities: \\
\begin{inparaenum}[a)]
\item\label{ca}
All connected components of $A_n\cap K_i$ are discs. \\
\item\label{cb}
The connected components of $A_n\cap K_i$ consist of one annulus $A_n'$
(topologically) isotopic to $A_n$ and discs.\\
\end{inparaenum}
Now let $i>i_1=i_1(\theta,\delta)$.
Then $\eta_{i-1}\vf_n$ is a non-zero smooth function on $A_n$ with compact support
in $A_n\cap K_{i}$ and Rayleigh quotient
\begin{equation*}
  R(\eta_{i-1}\vf_n) < \lambda_0(A_n) + \delta < \Lambda_D(S),
\end{equation*}
by \eqref{coray} and \eqref{pre2}.
Hence, if all the connected components of $A_n\cap K_i$ were discs,
we would have $R(\eta_{i-1}\vf_n)\ge\Lambda_D(S)$.
Therefore only case \ref{cb}) can occur.
Then the Rayleigh quotients of $\eta_{i-1}\vf_n$ on the discs of $A_n\cap K_i$
(on which $\eta_{i-1}\vf_n$ does not vanish) must be at least $\Lambda_D(S)$.
Hence the Rayleigh quotient of $\eta_{i-1}\vf_n$ on $A_n'$
must be less than $ \lambda_0(A_n) + \delta$.
In particular, $\lambda_0(A_n')\le\lambda_0(A_n)+\delta$.

Note that one of the boundary circles of $A_n'$ may only be piecewise smooth.
Therefore $A_n'$ may only be topologically isotopic to $A_n$.

Now we let $i=i_1+1$.
Since $A_n'\subseteq K_i$, we have the uniform area bound $|A_n'|\le|K_i|$.
This together with the above estimate on $\lambda_0(A_n')$ and \cref{lambdad}
implies that the length of shortest curves in $A_n'$,
which are freely homotopic to the boundary circles of $A_n'$ in $A_n'$,
is uniformly bounded with respect to $g$.
Then their length is also uniformly bounded with respect to $h$, by \eqref{normc}.
In particular, there are only finitely many isotopy types of $A_n'$
and, therefore also, of $A_n$.
Therefore we may pass to a subsequence so that all of them are isotopic.

By \cref{cheet} and since $K_i$ is compact,
we may pass to a further subsequence so that all $A_n \cap K_i$, and hence also all $A_n$,
contain a geodesic ball $B(x_0,\rho)$
such that $\varphi_n\ge\ve$ on $B(x_0,\rho)$ as claimed.
\end{proof}

\begin{proof}[Proof of \cref{annuli}]
By passing to a subsequence,
we may assume that the sequence of $A_n$ satisfies all the properties
from \eqref{pre2}, \eqref{pre4}, and \cref{isotyp}.

Choose a shortest (with respect to $h$) closed $h$-geodesic $c$
in the free homotopy class in $S$ of a generator of the fundamental group of $A_n$.
This is possible since the ends of $S$ are hyperbolic funnels with respect to $h$.
Note that $c$ does not depend on $n$ since all $A_n$ are isotopic.
We let $\hat S$ be the cyclic subcover of $\tilde S$ to which $c$ lifts
as a closed $\hat h$-geodesic $\hat c$,
where $\hat h$ denotes the lift of $h$ to $\hat S$.
Note that all annuli $A_n$ are isotopic to a small tubular neighbourhood of $c$.
Lifting the corresponding isotopies, we get lifts $\hat A_n\subseteq\hat S$ of the annuli $A_n$.
Note that $\hat A_n$ is the unique compact component of $\pi^{-1}(A_n)$
and that $\pi\colon\hat A_n\to A_n$ is a diffeomorphism
which is isometric with respect to $g$ and $h$ and their respective lifts $\hat g$ and $\hat h$.

Denote by $\hat x_n$ the lift of $x_0$ which is contained in $\hat A_n$.
If $\hat x_n$ stays at bounded distance to $\hat c$,
the arguments for the case of discs in the proof of \cref{discs} apply again
and lead to a contradiction since the fundamental group of $S$ is not cyclic
and $\lambda_0(\hat S)=\lambda_0(\tilde S)$.

Suppose now that $\hat x_n\to\infty$ in $\hat S$.
Let $(r,\theta) \colon \hat S \to \R\times(\R/\ell\Z)$ be Fermi coordinates
about $\hat c$,
where $\ell$ denotes the $h$-length of $c$ and $\hat c$, such that $\hat c=\{r=0\}$.
Then we have
\begin{equation*}
  \hat h = dr^2+\cosh^2(r)d\theta^2.
\end{equation*}
Since the $h$-length of shortest curves, $c_n$, in the free homotopy class
of the boundary curves of $A_n$ in $A_n$ is bounded by $\ell_0$,
there is a constant $r_0>0$ such that the lifts $\hat c_n$ of $c_n$ to $\hat A_n$
are contained in the region $\{|r|\le r_0\}$ of $\hat S$.

\begin{center}
\includegraphics[width=12cm]{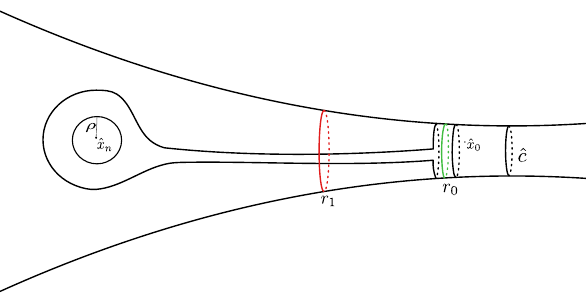}
\ \ \ \ The case $\hat x_n \to \infty$
\end{center}

Let $\hat\vf_n$ be the lift of $\vf_n$ to $\hat A_n$.
Then $\hat\vf_n$ is the ground state of $\hat A_n$ with respect to $\hat g$
and we have
\begin{equation}\label{lowerl2}
  \int_{B(\hat x_n,\rho)}\hat\vf_n^2
  = \int_{B(x_0,\rho)} \vf_n^2
  \ge \ve^2\vol B(x_0,\rho),
\end{equation}
by \cref{isotyp}.\ref{p6}.
Choose $j>i_1=i_1(\theta,\delta)$ such that
\begin{equation}\label{nabet}
  |\nabla\eta_j|^2 < \delta\ve^2\vol B(x_0,\rho).
\end{equation}
Since $\supp\eta_j\subseteq K_{j+1}$, the $h$-area of $\supp\eta_j$
is bounded by the $h$-area $|K_{j+1}|_h$ of $K_{j+1}$.
Now choose an $r_1>r_0$ such that the $\hat h$-area of either of the regions $-r_1\le r\le-r_0$
and $r_0\le r\le r_1$ in $\hat S$ is larger than $|K_{j+1}|_h$.
Finally, choose a cut off function $\chi$ on $\hat S$ such that $\chi=0$ on $\{|r|\le r_1\}$,
$\chi=1$ on $\{|r|\ge r_2\}$ for some $r_2>r_1$, and such that
\begin{equation}\label{nabch}
  |\nabla\chi|^2 < \delta\ve^2\vol B(x_0,\rho).
\end{equation}
Computing as in the proofs of \cref{comploc} and \cref{cheer} and with $\eta=\eta_j\circ\pi$, we get
\begin{equation} \label{upperh1_1}
	\begin{split}
  \int_{\hat A_n } | \nabla(\chi\eta\hat\varphi_n)|^2
  &= \int_{\hat A_n} \nabla(\chi^2\eta^2\hat\vf_n)\nabla\hat\vf_n
  + \int_{\hat A_n} \hat\vf_n^2 |\nabla(\chi\eta)|^2
  \\
  &= \lambda_0(A_n) \int_{\hat A_n} \chi^2\eta^2\hat\vf_n^2
  +  \int_{\hat A_n} \hat\vf_n^2 |\nabla(\chi\eta)|^2
  \\
  &\le \lambda_0(A_n) \int_{\hat A_n} \chi^2\eta^2\hat\vf_n^2
  + 4\delta\ve^2\vol B(x_0,\rho),
  	\end{split}
\end{equation}
where we use \eqref{nabet} and \eqref{nabch} for the passage to the last line.

Since $\hat x_n\to\infty$ and $B(x_0,\rho)\subseteq K_j$ by the choice of $j$,
$\chi=\eta=1$ on $B(\hat x_n,\rho)$ for all sufficiently large $n$.
Combining \eqref{lowerl2} and \eqref{upperh1_1}, we then get
\begin{equation} \label{rayl_est}
  R(\chi\eta\hat\vf_n) \le  \lambda_0(A_n) + 4\delta < \Lambda_D(S).
\end{equation}
On the other hand, $\supp(\chi\eta\hat\vf_n)$ is contained in the lift $B_n$
of $A_n\cap K_{j+1}$ to $\hat A_n$, intersected with $\{|r|\ge r_1\}$.
Now the $h$-area of $B_n$ is bounded by $|K_{j+1}|_h$
and $B_n$ contains $c_n$.
Hence $B_n$ does not contain loops freely homotopic to $c_n$
in the region $\{|r|\ge r_1\}$ of $\hat S$
since otherwise, by the uinqueness of $A_n'$,
it would contain one of the regions  $-r_1\le r\le-r_0$ or $r_0\le r\le r_1$.
Hence $B_n\cap\{|r|\ge r_1\}$ is a union of discs and, therefore,
the Rayleigh quotient of $\chi\eta\hat\vf_n$ has to be at least $\Lambda_D(S)$,
a contradiction to \eqref{rayl_est}.
It follows that the sequence of $\hat x_n$ is bounded.
\end{proof}

\begin{proof}[Proof of \cref{main3} in the case $\chi(S)<0$]
By Theorems \ref{discs} and \ref{annuli}, we have
\[ \min\big(\Lambda_D(S),\Lambda_A(S)\big) > \lambda_0(\tilde S) \]
for any complete and connected Riemannian surface of finite type with $\chi(S)<0$.
This implies the assertion of \cref{main3} in the case where $S$ is orientable
since then $S$ does not contain cross caps.

Assume now that $S$ is not orientable
and let $\Or(S)$ be the orientation covering space of $S$.
Let $C$ be a cross cap in $S$.
Then the lift of $C$ to $\Or(S)$ is an annulus $A$ in $\Or(S)$ with
\[ \lambda_0(C) \ge \lambda_0(A) \ge \Lambda_A(\Or(S)). \]
We conclude that $\Lambda_C(S)\ge\Lambda_A(\Or(S))>\lambda_0(\tilde S)$ as asserted.
\end{proof}

\section{Remarks, examples, and questions} \label{secr}

In this section, we collect some loose ends.
We start with a comment which gives another argument for calling $\Lambda$
the analytic systole.

\subsection{On the definition of $\Lambda$} \label{sused}
For complete and connected surfaces $S$ with infinite fundamental group,
an equivalent definition of the analytic systole is $\Lambda(S)=\inf_F \lambda_0(F)$,
where the infimum is taken over incompressible annuli and cross caps $F$
with smooth boundary in $S$: \\
\begin{inparaenum}[a)]
\item
For any disc $D$ with smooth boundary and any free homotopy class $[c]$ of closed curves in $S$,
there is an annulus $A$ with smooth boundary in $S$ containing $D$ whose soul belongs to $[c]$.
If $[c]$ is non-trivial, then $A$ is incompressible.
Moreover, $\lambda_0(A)\le\lambda_0(D)$ by the domain monotonicity of $\lambda_0$. \\
\item
For any compressible annulus $A$ in $S$ with smooth boundary,
there is a disc $D$ in $S$ whose boundary $\partial D$ is one of the boundary circles of $A$
such that $A\cup D$ is a disc in $S$ with smooth boundary, by the Schoenflies theorem,
and then a) applies. \\
\item
Cross caps only occur in the case where $S$ is not orientable.
The soul of a cross cap $C$ in $S$ is not homotopic to zero in $S$.
Since the fundamental group of $S$ is torsion free,
we get that $C$ is incompressible.
\end{inparaenum}

However, in view of our previous articles \cite{BMM, BMM1},
it is more natural to include discs into the definition.
Moreover, it is important in our analysis to handle the case of discs separately.

\subsection{On the essential spectrum}\label{susess}
The following result, Proposition 3.6 from \cite{BMM1} formulated for surfaces,
is probably folklore.
It shows that the essential spectrum of the Laplacian
only depends on the geometry of the underlying surface $S$ at infinity
and that the essential spectrum of the Laplacian is empty if $S$ is compact.

\begin{prop}\label{esp}
For a complete Riemannian surface $S$ with compact boundary (possibly empty),
$\lambda\in\R$ belongs to the essential spectrum of $\Delta$
if and only if there is a \emph{Weyl sequence} for $\lambda$, that is,
a sequence $\vf_n$ of smooth functions on $S$ with compact support such that \\
\begin{inparaenum}[1)]
\item
for any compact $K\subseteq S$,
$\supp\vf_n\cap K=\emptyset$ for all sufficiently large $n$; \\
\item
$\limsup_{n\rightarrow\infty}\|\vf_n\|_2>0$
and $\lim_{n\rightarrow\infty}\|\Delta\vf_n-\lambda\vf_n\|_2=0$.
\end{inparaenum}
\end{prop}

In work of Arne Persson and of Richard Froese and Peter Hislop,
the bottom of the essential spectrum of Laplacians and more general operators
has been characterized in the sense of \cref{esp} or, more specifically,
in the sense of \eqref{boess}; compare with \cite[Section 14.4]{HS}.

\begin{cor}\label{espb}
For a complete Riemannian surface $S$, we have
\begin{equation*}
  \lambda_{\ess}(S) = \lim_K\lambda_0(S\setminus K),
\end{equation*}
where $K$ runs over the compact subsets of $S$, ordered by inclusion.
\end{cor}

\begin{cor}\label{espc}
If $S$ is a compact Riemannian surface, then the spectrum of $S$ is discrete;
that is, $\lambda_{\ess}(S)=\infty$.
\end{cor}

\subsection{Surfaces with cyclic fundamental group} \label{susec}
In the (unnumbered) lemma on page 551 of \cite{Os},
Osserman establishes the following result in the special case of domains in the Euclidean plane.

\begin{lem}\label{os3}
Let $S$ be a complete Riemannian surface with boundary (possibly empty)
and $p$ be a point in the interior of $S$.
For sufficiently small $\ve>0$, let $S_\ve(p)=S\setminus B_\ve(p)$.
Then
\begin{equation*}
  \lambda_0(S_\ve(p))\to\lambda_0(S) \quad\text{as}\quad \ve\to0.
\end{equation*}
\end{lem}

The arguments in \cite{Os} also apply to the more general situation of \cref{os3}
and therefore we skip its proof.

\begin{prop}\label{equal1}
If $S$ is a complete Riemannian surface diffeomorphic to sphere or projective plane,
then $\Lambda(S)=0$.
\end{prop}

\begin{proof}
For $p\in S$ and $\ve>0$, $S_\ve(p)=S\setminus B_\ve(p)$
is a closed disc or cross cap, respectively, and hence
\begin{equation*}
  0 \le \Lambda(S) \le \inf_{p,\ve}\lambda_0(S_\ve(p)) = \lambda_0(S) = 0,
\end{equation*}
where we use \cref{os3} for the penultimate equality.
\end{proof}

\begin{prop}\label{equal3}
If S is a complete Riemannian surface diffeomorphic to disc (open or closed),
annulus (open, half-open, or closed), or cross cap (open or closed),
then $\lambda_0(S)=\Lambda(S)$.
\end{prop}

\begin{proof}
In each case, there exists an increasing sequence of closed discs,
annuli, or cross caps $F_n$, respectively,
which exhausts the interior $\mathring S$ of $S$.
Hence
\begin{equation*}
  \Lambda(S) = \lim\lambda_0(F_n) = \lambda_0(S),
\end{equation*}
by domain monotonicity and the definitions of $\Lambda(S)$ and $\lambda_0(S)$.
\end{proof}

\subsection{Examples}\label{susexa}
It follows from the constructions in \cite[Example 3.7]{BMM1}
that any non-compact and connected surface $S$ of finite type carries
complete Riemannian metrics of finite or infinite area with discrete spectrum,
that is, with $\lambda_{\ess}(S)=\infty$.
If the fundamental group of $S$ is not cyclic,
then $\Lambda(S)> \lambda_0(\tilde{S})$ for any such metric, by \cref{main3}.
In the following,
we extend some constructions from \cite{BMM1} slightly.

Let $F=\{(x,y)\mid x\ge0,\, y\in\R/L\Z\}$ be a funnel with the expanding
hyperbolic metric $dx^2+\cosh(x)^2 dy^2$.
Let $\kappa\colon\R\to\R$ be a monotonic smooth function with $\kappa(x)=-1$
for $x\le1$ and $\kappa(x)\to\kappa_\infty\in[-1,-\infty]$ as $x\to\infty$.
Let $j\colon\R\to\R$ solve $j''+\kappa j=0$
with initial condition $j(0)=1$ and $j'(0)=0$.
Then $j(x)\ge\cosh x$.
The funnel $F$ with Riemannian metric $g=dx^2+j(x)^2dy^2$
has curvature $K(x,y)=\kappa(x)\le-1$ and infinite area.
By comparison, the Rayleigh quotient with respect to $g$
of any smooth function $\vf$ with compact support in the part $\{x\ge x_0\}$
of the funnel is at least $-\kappa(x_0)/4$.

Let $S$ be a non-compact surface of finite type.
Endow $S$ with a hyperbolic metric which is expanding along its funnels as above.
Replace the hyperbolic metric on the funnels by the above Riemannian metric $g$.
Then the new Riemannian metric on $S$ is complete
with curvature $K\le-1$ and infinite area.
By \cref{esp} and by what we said above about the Rayleigh quotients,
the essential spectrum of the new Riemannian metric is contained in $[\kappa_\infty,\infty)$.
Choosing $\kappa$ such that $\kappa_\infty$ is larger than the first Dirichlet eigenvalue
$\lambda_0(D)$ of some smooth closed disc $D$ inside the surface yields the estimate
\begin{equation*}
  \lambda_{\ess}(S)>\lambda_0(D)>\lambda_0(\tilde S).
\end{equation*}
As a variation,
let $j$ solve $j''+\kappa j=0$ with initial condition $j(0)=1$ and $j(\infty)=0$.
Then $j'(0)\le-1$ and $j(x)\le\exp(-x)$.
The funnel $F$ with Riemannian metric $g=dx^2+j(x)^2dy^2$
has curvature $K(x,y)=\kappa(x)$ and finite area.
Again by comparison, the Rayleigh quotient with respect to $g$
of any smooth function $\vf$ with compact support in the part $\{x\ge x_0\}$
of the funnel is at least $-\kappa(x_0)/4$.

Let $S$ be a non-compact surface of finite type,
and choose $r>0$ such that $\coth(r)=-j'(0)$.
Now $S$ minus the parts $\{x\ge r\}$ of its funnels
carries hyperbolic metrics which are equal to $dx^2+j_0(x)^2dy^2$
along the parts $\{x<r\}$ of its funnels, where $j_0(x)=\sinh(r-x)/\sinh(r)$.
Then $j_0(x)=j(x)$ for $x<\min\{1,r\}$.
Hence any such hyperbolic metric,
restricted to $S$ minus the parts $\{x\ge\min\{1,r\}\}$ of its funnels,
when combined with $g$ along the funnels,
defines a smooth and complete Riemannian metric on $S$
which has curvature $K\le-1$ and finite area.
Choosing $\kappa$ and $D$ as in the first case,
we again have $\lambda_{\ess}(S)>\lambda_0(\tilde S)$.

\subsection{Generic metrics} \label{susgen}

In view of \cref{susec} and \cref{susexa} we are now prepared for the proof of \cref{proqua}.

\begin{proof}[Proof of \cref{proqua}]
For the first part observe that for any complete Riemannian surface $S$ of finite type, we have
\begin{equation*}
  \lambda_0(S) \le \lambda_0(\tilde S) \le \Lambda(S),
\end{equation*}
by \eqref{pomv} and \eqref{brooks}, respectively.
We conclude that
\begin{equation*}
  \lambda_0(S) = \lambda_0(\tilde S) = \Lambda(S)
\end{equation*}
for the surfaces considered in Propositions \ref{equal1} and \ref{equal3}. These surfaces are precisely the ones with cyclic fundamental group.

The second part follows immediately from \cref{susexa}, so we are only left with the
proof of the third part.

Let $S$ be a non-compact surface of finite type and $g$ be a complete Riemannian metric on $S$ with 
\begin{equation*}
\lambda_0(\tilde S,\tilde g)=\lambda_{\rm ess}(S,g).
\end{equation*}
By Theorem \ref{main1} we then have
\begin{equation*}
	\lambda_0(\tilde S,\tilde g)=\Lambda(S,g)=\lambda_{\rm ess}(S,g).
\end{equation*}
Now assume that $\lambda_{\rm ess}(S, g)>0$. For $n\ge1$,
let $F_n\subseteq S$ be a smooth closed disc, annulus or cross cap with
\begin{equation*}
	\lambda_0(F_n,g)
	< e^{1/n+1}\Lambda(S,g)
	= e^{1/n+1} \lambda_{\rm ess}(S,g).
\end{equation*}
Choose exhaustions of $S$ by compact subsets $K_n$ and $L_n$
and smooth functions $h_n$ such that, for all $n\ge1$,
\begin{equation*}
	F_n \subseteq \mathring K_n \subseteq K_n \subseteq \mathring L_n
\end{equation*}
and
\begin{equation*}
	e^{-1/n}\le h_n\le 1, \quad
	\text{$h_n=1$ on $K_n$}, \quad
	\text{$h_n=e^{-1/n}$ on $S\setminus L_n$.}
\end{equation*}
There exists a smooth function $f=f_t=f(t,x)$ on $(0,1]\times S$
with $f_{1/n} = h_n$ such that, for all $0<t\le1/n$,
\begin{equation*}
	\text{$f_t=1$ on $K_n$
		and $f_t=e^{-t}$ on $S\setminus L_n$}
\end{equation*}
and such that $f$ is monotonically decreasing in $t$.
Since $f_t=1$ on $K_n$ for $t\le1/n$ and the $K_n$ exhaust $S$,
$f$ can be smoothly extended to $[0,1]\times S$ by setting $f_0=1$.

Let $g_t=f_tg$.
Then $g_t$ is a smooth family of conformal metrics on $S$
and is a continuous curve of metrics with respect to the uniform distance.
For $t\le1/n$, we have
\begin{equation*}
	\Lambda(S,g_t) \le \lambda_0(F_n,g_t) =\lambda_0(F_n,g)
	< e^{1/n+1}\lambda_{\rm ess}(S,g).
\end{equation*}
Since the Dirichlet integral is invariant under conformal change in dimension two,
we obtain, for $1/n+1\le t\le1/n$,
\begin{equation*}
	\Lambda(S,g_t)
	< e^{1/n+1}\lambda_{\rm ess}(S,g)
	\le e^t \lambda_{\rm ess}(S,g)
	= \lambda_{\rm ess}(S,g_t).  
\end{equation*}
Invoking Theorem \ref{main1} we conclude that for all $t>0$ one has the inequality:
\begin{equation*}
	\lambda_0(\tilde S,\tilde g_t)
	< \Lambda(S,g_t)
	< \lambda_{\rm ess}(S,g_t).
\end{equation*}
It remains to show that the set of metrics $g$ on $S$ that satisfy the strict inequality
\begin{equation*}
	\Lambda(S, g) > \lambda_0(\tilde{S}, \tilde{g})
\end{equation*} 
is an open set in the uniform $C^\infty$ topology. This follows from the fact that two metrics that are close to each other in the uniform $C^\infty$ topology are  quasi-isometric by a quasi-isometry with quasi-isometry constant close to $1$. 
\end{proof}

\begin{rems} \label{rems_generic}
	1) Any two complete Riemannian metrics $g_0,g_1$ on $S$ of finite uniform distance are quasi-isometric.
	This implies, that there is a constant $C>0$ depending on the distance of $g_0$ to $g_1$ such that
	\begin{equation*}
	C^{-1}\lambda_{\rm ess}(S,g_0) \leq \lambda_{\rm ess}(S,g_1) \leq C \lambda_{\rm ess}(S,g_0). 
	\end{equation*}
	In particular, we have that
	\begin{itemize}
	\item[(i)] $\lambda_{\rm ess}(S,g_0)$ is finite iff $\lambda_{\rm ess}(S,g_1)$ is finite,
	\item[(ii)] $\lambda_{\rm ess}(S,g_0) = 0 $ iff $\lambda_{\rm ess}(S,g_1)=0.$
\end{itemize}	 
	2) The above construction can be extended to get metrics with
	\begin{equation*}
		\lambda_{\rm ess}(S,g_t)=e^t\lambda_{\rm ess}(S,g)
	\end{equation*}
	 for all $t\ge0$. \\
	3) If $S$ is non-compact, any complete hyperbolic metric on $S$ satisfies 
	$\lambda_0(\tilde S,\tilde g)=\lambda_{\rm ess}(S,g)=1/4.$\\
	4) If $S$ is non-compact, any complete Riemannian metric on $S$,
	which is in zeroth order asymptotic to a flat  cylinder $\R/L\Z \times [0,\infty),$
	has $\lambda_0(\tilde{S},g)=\lambda_{\rm ess}(S,g)=0.$
\end{rems}

\subsection{Problems and questions} \label{suseq}
We collect some questions naturally arising in view of our results.
Let $S$ be a compact and connected Riemannian surface with negative Euler characteristic. \\
\begin{inparaenum}[1)]
\item\label{opde}({\sc Optimal design})
For a given compact subsurface $T$ of $\mathring S$
with smooth boundary $\partial T\ne\emptyset$,
we may consider the constant
\[ \Lambda_T(S) = \inf_F\lambda_0(F),\]
where $F$ runs over all subsurfaces of $S$ isotopic to $T$.
The analytic systole is an infimum over such constants.
It is interesting to ask for estimates of $\Lambda_T(S)$.
The infimum is probably achieved by degenerate $F$,
where $\partial F$ is mapped onto a graph $\Gamma$ in $S$
such that $S\setminus\Gamma$ is diffeomorphic to the interior of $T$.
In fact, for any $F$ isotopic to $T$, there is a graph $\Gamma$ in $S$
such that $F\subseteq S\setminus\Gamma$ and such that $S\setminus\Gamma$
is isotopic to the interior of $F$.
Hence, by domain monotonicity,
$\Lambda_T(S)$ is the infimum over all $\lambda_0(S\setminus\Gamma)$,
where $\Gamma$ runs through such graphs.
What are the optimal graphs?
This circle of problems is related to the work of Helffer, Hoffman-Ostenhof, and Terracini \cite{HHT}.

\item\label{rigi}({\sc Rigidity})
The inequality $\lambda_{-\chi(S)}(S)>\Lambda(S)$, mentioned in the introduction,
raises the question whether there is another natural
geometric constant $\Lambda'(S)>\Lambda(S)$,
where we only have the weak inequality $\lambda_{-\chi(S)}\ge\Lambda'(S)$
and where equality occurs only for a distinguished class of Riemannian metrics.

\item\label{rigi1}({\sc Another rigidity})
The last part of \cref{proqua} suggests that hyperbolic metrics on non-compact surfaces of finite 
type are among a small collection of metrics that satisfy 
\begin{equation*}
\Lambda(S) = \lambda_0(\tilde{S}) = \lambda_{ess}(S).	
\end{equation*}
It would be interesting to see what other implications this equality has on the metric. 
If we rescale the metric by a function $f \colon S \to (0,1]$ which is $1$ outside a compact set, then
$\lambda_0(\tilde S)$ can only increase, while $\lambda_{\rm ess}(S)$ remains unaffected. Using our
main theorem we can see that the new metric also satisfy the above equality. Hence one can not have a 
rigidity among all smooth metrics. Also, observe that points 1)\ and 4)\ of \cref{rems_generic} imply 
that there is no such rigidity for metrics with $\lambda_{\rm ess}(S)=0.$

\item\label{higher_dim}({\sc Higher dimensions})
All our definitions extend in a natural way to higher dimensional manifolds.
For instance, we may define the analytic systole of an $n$-dimensional manifold $M$
by $\Lambda(M)= \inf_\Omega\lambda_0(\Omega)$,
where $\Omega$ runs over all tubular neighborhoods
about essential simple loops in $M$.
By \eqref{pomv}, we have $\Lambda(M)\ge\lambda_0(\tilde M)$.
One may ask whether the strict inequality holds true under reasonable assumptions on $M$.
Our methods seem to be too weak to adress this question.
\end{inparaenum}

\appendix
\section{On $\lambda_0$ under coverings} \label{appendix}

In \cite{Br2}, Brooks states that,
for a Riemannian covering $\pi\colon\hat M\to M$ of complete Riemannian manifolds without boundary,
the bottom of the spectrum remains unchanged provided the covering is normal with amenable covering group and that $M$ has finite topological type, that is,
that $M$ is the union of finitely many simplices.
We use the corresponding result in the case where the covering is normal with cyclic fundamental group,
but where the boundaries of $\hat M$ and $M$ may not be empty.
In fact, in \cite{BMM1} we also claim that the results there remain true for Schr\"odinger operators $\Delta+V$ with non-negative potential $V$.

By the proof of \cite[Theorem (2.1)]{Su} or \cite[Theorem 7]{CY},
the bottom $\lambda_0(M,V)$ of the spectrum of a Schr\"odinger operator $\Delta+V$ on a complete and connected  Riemannian manifold $M$ with boundary (possibly empty) with non-negative potential $V$.
is the top of the positive spectrum of $\Delta+V$.
Now for a Riemannian covering $\pi\colon\hat M\to M$ of complete and connected Riemannian manifolds with boundary (possibly empty) and non-negative potentials $V$ and $\hat V=V\circ\pi$,
the lift of a positive $\lambda$-eigenfunction of $\Delta+V$ on $M$ to $\hat M$ is a positive $\lambda$-eigenfunction of $\Delta+\hat V$.
Therefore
\begin{equation}\label{pomv}
  \lambda_0(M,V) \le \lambda_0(\hat M,\hat V)
\end{equation}
in this situation.
Since the lift of a square integrable function on $M$ to $\hat M$ is square integrable if the covering is finite,
the reverse inequality holds for such coverings,
but does not hold in general.

\begin{thm}\label{brooks2}
Let $\pi\colon\hat M\to M$ be a normal Riemannian covering of complete and connected Riemannian manifolds with boundary (possibly empty)
with infinite cyclic covering group.
Let $V\colon M\to\R$ be a smooth non-negative function and set $\hat V=V\circ\pi$.
Then $\lambda_0(M,V)=\lambda_0(\hat M,\hat V)$.
\end{thm}

The case of the standard Laplacian corresponds to the case $V=0$.
Note that we do not need to assume that $M$ has finite topological type in the sense of Brooks.

\begin{proof}[Proof of \cref{brooks2}]
By \eqref{pomv}, we have $\lambda_0(M,V)\le\lambda_0(\hat M,\hat V)$.
To show the reverse inequality, let $\ve>0$ and $\vf$ be a smooth function on $M$ with compact support in the interior of $M$ and Rayleigh quotient
\begin{equation*}
  R(\vf) = \int_M \{ |\nabla\vf|^2 + V\vf^2 \} \big/ \int_M\vf^2 < \lambda_0(M,V)+\ve.
\end{equation*}
The strategy is now to cut off the lift $\hat\vf=\vf\circ\pi$ of $\vf$ to $\hat M$ conveniently so that the Rayleigh quotient of the new function is bounded by $R(\vf)+\ve$.

We note first that the covering $\pi_0\colon\R\to\R/\Z$ is universal.
Hence the covering $\pi$ is the pull back of $\pi_0$ by a smooth map $f\colon M\to\R/\Z$.
Without loss of generality, we may assume that $[0]\in\R/\Z$ is a regular value of $f$.
Then $f^{-1}([0])$ is a smooth hypersurface of $M$.

Up to covering transformation, there is a unique lift $\hat f\colon\hat M\to\R$ of $f$.
Then $\pi^{-1}(f^{-1}([0]))\subseteq \hat M$ is the union of the smooth hypersurfaces $\hat f^{-1}(k)$, $k\in\Z$.
Moreover, $\hat f^{-1}([k,k+1])$ is a smooth fundamental domain for the action of $\Z$ on $\hat M$, for all $k\in\Z$, and $\supp\hat\vf\cap\hat f^{-1}([a,b])$ is compact, for all $a\le b$.

Let $\eta_0$ be a non-negative smooth function on $\hat M$ which is positive on $\hat f^{-1}([0,1])$
and which has support in $\hat f^{-1}([-1,2])$.
Set $\eta_k=\eta_0\circ\lambda_k$,
where $\lambda_k$ denotes the action of $k\in\Z$ on $\hat M$,
and $\zeta_k=\eta_k/\sum_{j\in\Z}\eta_j$.
Note that the sum in the denominator on the right is well defined since it is locally finite.
Then $(\zeta_k)$ is a partition of unity on $\hat M$ such that $\zeta_k=\zeta_0\circ\lambda_k$.
In particular, since $\zeta_0$ has support in $\hat f^{-1}([-1,2])$
and $\supp\hat\vf\cap\hat f^{-1}([-1,2])$ is compact,
there is a uniform bound $|\nabla\zeta_k|\le C$ on $\supp\hat\vf$.
Therefore
\begin{equation*}
  \chi_k = \sum_{-1\le j\le k+1} \zeta_j
\end{equation*}
is a smooth cut-off function on $\hat M$ with values in $[0,1]$ which is equal to $1$ on $f^{-1}([0,k])$,
has support in $f^{-1}([-2,k+2])$,
and gradient bounded by $3C$ on $\supp\hat\vf$.
We conclude that
\begin{equation*}
  \int_{\hat M} (\chi_k\hat\vf)^2 \ge k \int_{M} \vf^2,
  \quad
  \int_{\hat M} \hat V(\chi_k\hat\vf)^2 \le (k+4) \int_M V\vf^2,
\end{equation*}
and, using Young's inequality,
\begin{align*}
    \int_{\hat M} |\nabla(\chi_k\hat\vf)|^2
    &\le (1+\delta) \int_{\hat M} |\nabla\vf|^2 + \left(1+\frac{1}{\delta} \right)\int_{\supp \nabla \chi_k } |\nabla\chi_k|^2\hat\vf^2 \\
    &\le  (1+\delta)(k+4) \int_M |\nabla\vf|^2 + \left(1+\frac{1}{\delta} \right)36C^2 \int_M \vf^2,
\end{align*}
since $\supp \nabla \chi_k \subseteq \hat f^{-1}([-2,0]\cup[k,k+2]).$
Choosing $\delta$ small enough, we hence find that $\chi_k\hat\vf$ is a smooth function on $\hat M$ with compact support such that $R(\chi_k\hat\vf)<R(\vf)+\ve$ for all sufficiently large $k.$
\end{proof}



\end{document}